\documentclass[12pt, reqno]{amsart}
\usepackage{amsmath,amssymb,amsthm,graphicx,mathrsfs,url} % Required for inserting images
\usepackage{amsmath,amssymb,amsthm,tikz}
\usetikzlibrary{decorations.pathreplacing}
\usepackage[toc,page]{appendix}
\usepackage{comment}
\usepackage[dvipsnames]{xcolor}
\usepackage[margin=0.9in]{geometry}
\usepackage[colorlinks=true, citecolor=Green]{hyperref}
\usepackage[english]{babel}
\usepackage[T1]{fontenc}
\usetikzlibrary{arrows.meta,calc}
\usepackage{amsthm}
\usepackage{ragged2e} % pour \RaggedLeft
\usepackage{etoolbox}

\usetikzlibrary{decorations.pathreplacing}

\newcommand{\tr}{\mathrm{tr}}
\newcommand{\eps}{\varepsilon}

\newtheorem{theorem}{Theorem}[section]
\newtheorem{lemma}[theorem]{Lemma}
\newtheorem{assumption}[theorem]{Assumption}

\newtheorem{corollary}[theorem]{Corollary}
\newtheorem{definition}[theorem]{Definition}
\newtheorem{remark}[theorem]{Remark}

\newtheorem{proposition}[theorem]{Proposition}
\newtheorem*{theoremfrench*}{Théorème}
\newtheorem*{theorem*}{Theorem}

\title[Unstable entropy for Anosov diffeomorphisms on tori]{Unstable entropy for Anosov diffeomorphisms on tori isotopic to Anosov linear maps}
\author{Tristan Humbert}
\begin{document}
\maketitle
\begin{abstract}
Let $f$ be an Anosov diﬀeomorphism on the $d$-torus which is strongly partially hyperbolic with expanding center. If $d=3$ or if $f$ is isotopic to its linear part, where the isotopy is contained in the set of partially hyperbolic diﬀeomorphisms, we construct systems of strong-unstable and center-stable Margulis measures which are  holonomy invariant. As a consequence, we obtain a characterization of measures of maximal unstable entropy in terms of their conditional measures along strong-unstable leaves. Moreover, we show that the unstable topological entropy, a measure of maximal unstable entropy and the Margulis system can be retrieved from the Pollicott-Ruelle spectrum of $f$. 
\end{abstract}

\maketitle
\section{Introduction}
\subsection{Anosov diffeomorphism with a partially hyperbolic splitting.} Let $d\geq 3$ be an integer. In this paper, we consider $f\in \mathrm{Diff}^\infty(\mathbb T^d)$ a smooth topologically transitive diffeomorphism of the $d$-torus $\mathbb T^d$ for which there exists a continuous and $df$-invariant splitting of the tangent bundle $T(\mathbb T^d)=E_s\oplus E_c\oplus E_u$ where $E_s$ (resp. $E_c$ and $E_u$) is uniformly contracted (resp. are uniformly expanded) by $df$ and where $E_u$ dominates $E_c$. In other words, $f$ can be seen either  as an \emph{Anosov} diffeomorphism with stable (resp. unstable) bundle given by $E_s$ (resp. $E_{cu}:=E_c\oplus E_u$) or as a partially hyperbolic diffeomorphism with (uniformly expanding) center bundle $E_c$ and strong-unstable bundle given by $E_u$. Let $d_\bullet=\mathrm{dim}(E_\bullet)$ for $\bullet=s,c,u$. We will denote by $\mathcal A^\infty_{d_s,d_c,d_u}(\mathbb T^d)\subset  \mathrm{Diff}^\infty(\mathbb T^d)$ the set of such diffeomorphisms. We make the following assumption on $f$.
\begin{assumption}
\label{assump}
The bundles $E_{cs}$ and $E_u$ are orientable and $f$ preserves the orientation of $E_{cs}$ and $E_u$.
\end{assumption} We make this assumption for convenience of the exposition. Note that, up to passing to a double cover and considering $f^2$, the above assumption can always be achieved.
\subsection{Unstable entropy} Given $f\in  \mathcal A^\infty_{d_s,d_c,d_u}(\mathbb T^d)$ and a $f$-invariant probability measure $\mu$, we can consider the \emph{unstable topological entropy} (resp. \emph{unstable metric entropy} with respect to $\mu$) denoted by $h^u_\mathrm{top}(f)$ (resp. $h_\mu^u(f)$). These invariants, introduced in \cite{SaXi,HuHuWu}, correspond to partial entropies measuring the complexity of the dynamics of $f$ along the strong-unstable foliation, see \S \ref{sec:entropy} for more details. The two notions of entropy are linked by a \emph{variational principle} \cite[Theorem D]{HuHuWu}, 
\begin{equation}
\label{eq:varia111}
h^u_{\mathrm{top}}(f)=\sup_{\mu \in \mathcal P_f(\mathbb T^d)}h^u_\mu(f),
\end{equation}
where $\mathcal P_f(\mathbb T^d)$ denotes the set of $f$-invariant probability measures on $\mathbb T^d$.
Of particular interest is the study of \emph{measures of maximal $u$-entropy} or \emph{$u$-MMEs} corresponding to probability measures attaining the supremum in \eqref{eq:varia111}. Their set, which will be denoted by $\mathcal{M}^u(f)$, is non-empty, convex and compact. Moreover, the extreme points are precisely the ergodic measures of maximal $u$-entropy, see \cite[Theorem B]{HuWuZh}. 
\subsection{Margulis system of measures}
Let $A\in \mathrm{SL}(d,\mathbb Z)\cap \mathcal A^\infty_{d_s,d_c,d_u}(\mathbb T^d)$ and denote  by $\mathcal A^\infty_A(\mathbb T^d)$ the set of $f\in \mathcal A^\infty_{d_s,d_c,d_u}(\mathbb T^d)$ which are in the connected component of $A$ in the space of partially hyperbolic diffeomorphisms. Note that for any $f\in \mathcal A^\infty_A(\mathbb T^d)$, the matrix $A$ coincides with the linearization (i.e. the action on homology) of $f$. 
 The diffeomorphism $f$ is \emph{dynamically coherent}, that is, the center-stable bundle $E_{cs}:=E_c\oplus E_s$ integrates to an invariant foliation $(\mathcal W^{cs}(x))_{x\in \mathbb T^d}$, see \cite{FPS}. We refer to \S \ref{sec:diffeo} for the detailed definitions and more properties of $ \mathcal A^\infty_{d_s,d_c,d_u}(\mathbb T^d)$ and of $\mathcal A^\infty_A(\mathbb T^d)$.
 
 Note that the strong-unstable bundle $E_u$ always integrates to an invariant foliation which we denote by $(\mathcal W^{u}(x))_{x\in \mathbb T^d}$, see \S  \ref{sec:mani} for a more detailed discussion about invariant manifolds of partially hyperbolic diffeomorphisms.
In order to study $u$-MMEs, we will construct two systems of measures $(\mu^u_x)_{x\in \mathbb T^d}$ (resp. $(\mu^{cs}_x)_{x\in \mathbb T^d}$) supported on these manifolds.

A \emph{system of $u$-measures} $($resp. \emph{$cs$-measures$)$} is a family of Borelian measures $(\mu^u_x)_{x\in \mathbb T^d}$ (resp. $(\mu^{cs}_x)_{x\in \mathbb T^d})$ on strong-unstable manifolds $(\mathcal W^u(x))_{x\in \mathbb T^d}$ (resp. center-stable manifolds $(\mathcal W^{cs}(x))_{x\in \mathbb T^d}$), see \S \ref{sec:system} for a precise definition. For $p,q\in \mathbb T^d$, close enough, we denote by $\mathrm{Hol}^u_{p,q}:\mathcal W^{cs}(p)\to \mathcal W^{cs}(q)$ (resp. $\mathrm{Hol}^{cs}_{p,q}:\mathcal W^{u}(p)\to \mathcal W^{u}(q)$) the local $u$-holonomy (resp. $cs$-holonomy), see \S \ref{sec:product} for a review of this notion. 
 
To state our results, let us define the \emph{center Jacobian}. There is $\alpha>0$ such that
$$
J^c_f\in C^{\alpha}(\mathbb T^d,\mathbb R),\ \forall x\in \mathbb T^d, \quad J_f^c(x):=-\log| \mathrm{det}(df_x|_{E_c(x)})|.
$$
Recall that as $f$ is Anosov, there is a unique equilibrium state $\mu_{J^c_f}$ associated to the potential $J_f^c$. 
We will denote by $\lambda^u(A):=\mathrm{det}(A|_{E_u^A})$ the determinant of the matrix $A$ on the strong-unstable bundle $E_u^A$ of $A$. Note that $\lambda^u(A)$ is equal to the product of the $d_u$ eigenvalues of $A$ with the largest modulus. When $f$ satisfies Assumption \ref{assump}, $\lambda^u(A)$ also corresponds to the spectral radius of $f_{*,d_u}:H_{d_u}(\mathbb T^d,\mathbb R)\to H_{d_u}(\mathbb T^d,\mathbb R)$, the action of $f$ on the homology group of degree $d_u$, i.e., $\lambda^u(A)=\max\{|\mu|, \mu \in \mathrm{Spec}(f_{*,d_u})\}$.

In this paper, we compute $h^u_{\mathrm{top}}(f)$ and show that $\mu_{J^c_f}$ is a $u$-MME for $f\in\mathcal A^\infty_A(\mathbb T^d)$. As a key step in the proof, we construct two \emph{Margulis systems} adapted to the partially hyperbolic splitting which give a characterization of $u$-MMEs in terms of their conditional measures along strong-unstable manifolds.
\begin{theorem}
\label{main theo}
 Let $f\in\mathcal A^\infty_{d_s,d_c,d_u}(\mathbb T^d)$ satisfying Assumption \ref{assump}. Suppose that:
 \begin{itemize}
 \item either $d=3$;
\item or $f\in A^\infty_A(\mathbb T^d)$ for some $A\in \mathrm{SL}(d,\mathbb Z)\cap \mathcal A^\infty_{d_s,d_c,d_u}(\mathbb T^d)$.
 \end{itemize}
 Then we have
 \begin{equation}
\label{eq:equality}
e^{h^u_{\mathrm{top}}(f)}=\lambda^u(A)=e^{h^u_{\mathrm{top}}(A)}=e^{P(J^c_f)}, \quad \mu_{J^c_f}\in \mathcal{M}^u(f).
\end{equation}
Moreover, there exists a system of $u$-measures $(\mu^{u}_x)_{x\in \mathbb T^d}$ and a system of $cs$-measures $(\mu^{cs}_x)_{x\in \mathbb T^d}$ such that:
\begin{itemize}
\item  the systems are conformal with respect to $f$:
\begin{equation}
\label{eq:conf}
\forall x\in \mathbb T^d, \quad  f_*\mu^{cs}_x=e^{h^u_{\mathrm{top}}(f)}\mu^{cs}_{f(x)}, \quad  f^*\mu^u_x=e^{h^u_{\mathrm{top}}(f)}\mu^u_{f^{-1}(x)};
\end{equation} 
\item the systems are holonomy invariant:
\begin{equation}
\label{eq:holhol}
 (\mathrm{Hol}^{u}_{y,w})_*\mu^{cs}_y= \mu^{cs}_w, \quad (\mathrm{Hol}^{cs}_{x,z})_*\mu^u_x= \mu^u_z,
\end{equation}
for any points $x,z$ $($resp. $y,w)$ for which the center-stable $($resp. strong-unstable$)$ holonomy is well defined. 
\end{itemize}
Finally, for any $\nu\in \mathcal P_f(\mathbb T^d)$, we have $\nu\in \mathcal M^u(f)$ if and only if the conditionals of $\nu$ along $\mathcal W^u$ are given by $(\mu^u_x)_{x\in \mathbb T^d}$ $($up to a constant rescaling$)$, $\nu$-a.e.
\end{theorem}
 Theorem \ref{main theo} recovers \cite[Theorems 1 and 2]{Humu} for diffeomorphisms of the $3$-torus and extends these results to higher dimensional tori when $f$ is in the connected component of its action on homology. The overall strategy of the proof follows the ideas developed in \cite{Humu} although we simplify some arguments, see \S \ref{sec:comp} for further discussion.
\begin{remark}
As a byproduct of the proof, we show that $\mu_{J^c_f}$ has a $cs/u$ product structure. For any rectangle $R\subset \mathbb T^d$ and $q\in R$, there is $c_R>0$ such that $($see \S \ref{sec:product}  for the notation$)$:
\begin{equation}
\label{eq:localproduct}
\forall \phi\in L^1(R,\mu_{J^c_f}),\quad \mu_{J^c_f}(\phi)=c_R\int_{\mathcal W^{cs}_R(q)}\int_{\mathcal W^{u}_R(y)}\phi(z)d\mu^{u}_y(z)d \mu^{cs}_q(y). 
\end{equation}
\end{remark}
\subsection{Pollicott-Ruelle resonances}
As in \cite{Humu}, the Margulis systems and unstable topological entropy are related to \emph{Pollicott-Ruelle resonances} of $f$. For $0\leq k \leq d$, let $\mathcal D'_{k}(\mathbb T^d)$ denote the space of $k$-currents, that is, the dual space to smooth $(d-k)$-forms. Let us define a splitting of the cotangent bundle, dual to the Anosov splitting of $f$,
\begin{equation}
\label{eq:dual1}
T^*(\mathbb T^d)=E_s^*\oplus E_{cu}^*,\quad E_s^*(E_s)=0, \ E_{cu}^*(E_{cu})=0.
\end{equation}
One can associate to an Anosov map $f$ a discrete spectrum $\mathrm{Res}_k(f)\subset \mathbb C$ : the \emph{Pollicott-Ruelle resonances}. Recall that $\lambda \in \mathrm{Res}_k(f)$ if and only if
\begin{equation}
\label{eq:res}
\exists u\in \mathcal D'_k(\mathbb T^d)\setminus\{0\},\quad f_*u=\lambda u, \ \mathrm{WF}(u)\subset E_{cu}^*. 
\end{equation}
Here, $\mathrm{WF}(u)\subset T^*\mathbb T^d\setminus \{0\}$ denotes the wavefront set of $u$, see \cite[Chapter VIII]{Hor}. In this case, the current $u$ is called a \emph{resonant state} associated to the resonance $\lambda$. We have a dual notion of \emph{co-resonant state}:
\begin{equation}
\label{eq:cores}
\exists u\in \mathcal D'_{d-k}(\mathbb T^d)\setminus\{0\},\quad f^*u=\lambda u, \ \mathrm{WF}(u)\subset E_{s}^*. 
\end{equation}
 We say that $\lambda \in \mathbb R_+$ is the \emph{first resonance for the action on $k$-forms} if $\lambda\in \mathrm{Res}_k(f)$ and $\mathrm{Res}_k(f)\subset \{z\in \mathbb C, |z|\leq \lambda\}.$ 

\begin{theorem}[First Pollicott-Ruelle resonance]
\label{theores}
Let $f\in\mathcal A^\infty_{d_s,d_c,d_u}(\mathbb T^d)$ satisfying Assumption \ref{assump}. Suppose that:
 \begin{itemize}
 \item either $d=3$;
\item or $f\in A^\infty_A(\mathbb T^d)$ for some $A\in \mathrm{SL}(d,\mathbb Z)\cap \mathcal A^\infty_{d_s,d_c,d_u}(\mathbb T^d)$.
 \end{itemize}
 Then $e^{h^u_{\mathrm{top}}(f)}$ is the first Pollicott-Ruelle resonance for the action on $(d_s+d_c)$-forms. Moreover, $e^{h^u_{\mathrm{top}}(f)}$ is the only resonance on the critical circle $\{ z\in \mathbb C, |z|=e^{h^u_{\mathrm{top}}(f)}\}$, it has no Jordan block and the space of resonant $($resp. co-resonant$)$ states is one-dimensional.

The Margulis system $(\mu_{x}^u)_{x\in \mathbb T^d}$ defines a $d_u$-current $\mu^u$ which is a co-resonant state associated to the first resonance $e^{h^u_{\mathrm{top}}(f)}$, that is,
\begin{equation}
\label{eq:muucores}
f^*\mu^u=e^{h^u_{\mathrm{top}}(f)}\mu^u,\quad \mathrm{WF}(\mu^u)\subset E_s^*.
\end{equation}Finally, the trace of the spectral projector at $e^{h^u_{\mathrm{top}}(f)}$ is given by $\mu_{J^c_f}$.
\end{theorem}
By Theorem \ref{theores}, we can interpret the exponential of the unstable topological entropy as an eigenvalue for the action of $f$ on specially designed functional spaces (the anisotropic spaces). Moreover, the $u$-system $(\mu^{u}_x)_{x\in \mathbb T^d}$ identifies with the corresponding (dual) eigenfunction and the $u$-MME $\mu_{J^c_f}$ can be read as the trace of the spectral projector.  

Theorem \ref{theores} recovers \cite[Theorem 3]{Humu} for diffeomorphisms of the $3$-torus and generalizes the result to higher dimensional tori when $f$ is in the connected component of its action on homology.
\subsection{Applications}
We record some consequences of the existence of a strong-unstable Margulis system.

 In a recent survey \cite{Ta}, Tahzibi asked whether for an Anosov diffeomorphism which is partially hyperbolic with expanding center, a measure of maximal $u$-entropy is necessarily a measure of maximal entropy (or MME). In \cite[Corollary 1.1]{Humu}, the author proved, using \cite{GaSh}, that when $d=3$, the $u$-MME $\mu_{J^c_f}$ is a MME if and only if $E_u$ and $E_s$ are jointly integrable. Moreover, this condition fails for a $C^1$-open and $C^\infty$-dense subset of diffeomorphisms.
 
For higher dimensional tori, we can show the following result.  
\begin{corollary}
\label{cor1}
Let $A\in \mathrm{SL}(d,\mathbb Z)\cap \mathcal A^\infty_{d_s,d_c,d_u}(\mathbb T^d)$ satisfying Assumption \ref{assump}. Then 
$$\{f\in \mathcal A^\infty_A(\mathbb T^d)\mid \mu_{J^c_f} \text{ is not a MME for } f\}\subset \mathcal A^\infty_A(\mathbb T^d) $$
is a $C^1$-open and $C^\infty$-dense subset.
\end{corollary}
We are not able to generalize the characterization obtained in \cite[Corollary 1.1]{Humu} in full generality. However, using the recent work of Gogolev-Shi \cite[Theorem 1.3]{GS23}, we can show the following.
\begin{corollary}
\label{cor2}
Let $A\in \mathrm{SL}(d,\mathbb Z)\cap \mathcal A^\infty_{d_s,1,d_u}(\mathbb T^d)$ with a $1$-dimensional center bundle. Suppose furthermore that $A$ is generic and satisfies Assumption \ref{assump}. Then there is a $C^1$-neighborhood $\mathcal U$ of $A$ in $\mathcal A^\infty_{d_s,1,d_u}(\mathbb T^d)$ such that 
$$\forall f\in \mathcal U,\quad \mu_{J^c_f} \text{ is a MME for } f \iff E_u \text{ and } E_s \text{ are jointly integrable.} $$
\end{corollary}
The precise meaning of generic (which will not be used in this paper) can be found in the introduction of \cite{GS23}.

Another important class of invariant measures is given by \emph{$u$-Gibbs measures}. They correspond to invariant measures whose conditionals along strong-unstable manifolds are absolutely continuous with respect to the Lebesgue measure. Equivalently, they correspond to $u$-equilibrium states associated to the \emph{strong-unstable Jacobian}:
$$
J^u_f\in C^{\alpha}(\mathbb T^d,\mathbb R),\ \forall x\in \mathbb T^d, \quad J_f^u(x):=-\log| \mathrm{det}(df_x|_{E_u(x)})|,
$$ see \cite[Theorem C]{HuWuZh} for more details. Using \cite[Theorem A]{GS25}, we obtain the following result.
\begin{corollary}
\label{cor3}
Let $f\in\mathcal A^\infty_{d_s,d_c,d_u}(\mathbb T^d)$ satisfying Assumption \ref{assump} such that $\mathcal W^u$ is minimal. Suppose that:
 \begin{itemize}
 \item either $d=3$;
\item or $f\in A^\infty_A(\mathbb T^d)$ for some $A\in \mathrm{SL}(d,\mathbb Z)\cap \mathcal A^\infty_{d_s,d_c,d_u}(\mathbb T^d)$.
 \end{itemize}
Suppose that $\mu_{J^c_f}$ is a $u$-Gibbs measure. Then $J^u_f$ is cohomologous to $-\log(\lambda^u(A))$.
\end{corollary}
When $d=3$, the minimality of $\mathcal W^u$ for any $f\in \mathcal A^\infty_{1,1,1}(\mathbb T^3)$ was announced by Avila-Crovisier-Eskin-Potrie-Wilkinson-Zhang. In this setting, Corollary \ref{cor3} was already obtained in \cite[Theorem C]{GS25}. In higher dimension, the minimality of the strong-unstable foliation is a difficult and active field of research and we refer to \cite{ACW25} for more details. In particular, \cite[Theorem B]{ACW25} shows that if $f\in \mathcal A^\infty_{d_s,d_c,d_u}(\mathbb T^d)$ has minimal strong-unstable foliation and $E_u$ and $E_s$ are \emph{not} jointly integrable, then $f$ is \emph{robustly minimal}, that is, it admits a $C^1$-neighborhood of diffeomorphisms for which the strong-unstable foliation remains minimal.

We also note that by \cite[Proposition 6]{Gog08}, if $A\in \mathrm{SL}(d,\mathbb Z)\cap \mathcal A^\infty_{d_s,d_c,d_u}(\mathbb T^d)$ is irreducible with a simple and real leading eigenvalue, then $E_u^A$ is dense. We explain in \S \ref{sec:last} how to obtain non-linear examples of $f\in\mathcal A^\infty_{d_s,d_c,d_u}(\mathbb T^d)$ having minimal strong-unstable foliation from such a CAT map.

\subsection{Comparison with earlier works}
\label{sec:comp}
As explained above, Theorems \ref{main theo} and \ref{theores} generalize the results of \cite{Humu} to higher dimensional tori for any $f\in \mathcal A^\infty_A(\mathbb T^d)$, where $A\in \mathrm{SL}(d,\mathbb Z)\cap \mathcal A^\infty_{d_s,d_c,d_u}(\mathbb T^d)$. As such, the current paper fits into the general program of extending the construction of Margulis systems from the uniformly hyperbolic setting to some partially hyperbolic systems. These results are also new instances of partially hyperbolic systems for which some of the objects of thermodynamic formalism can be retrieved from the spectral theory associated to $f$. We refer to \cite[\S 1.5]{Humu} for a review of the existing literature on these subjects as well as for a comparison with other existing works in the field. 

In the rest of this section, we explain how our current work differs from the earlier paper \cite{Humu}. The general philosophy of the two papers are the same, namely, we construct dynamically a $u$-Margulis system which we interpret as a Pollicott--Ruelle coresonant state for the action of $f$ on $(d_s+d_c)$-forms. Analyzing the first resonance for this action using analytical tools, we show the equality between the unstable entropy and pressure on the center Jacobian. This in turn shows that the associated equilibrium state is a $u$-MME for the partially hyperbolic diffeomorphism.

In \cite{Humu}, we used the \emph{quasi-isometry} of the strong-unstable foliation proved when $d=3$ in \cite{Po} to construct a first $u$-system using a compactness argument \emph{à la} Margulis. In order to construct a measurable $u$-system from this first system, we used ideas coming from dimension theory \cite{CPZ20}. We note that recently, the work \cite{BSYY26} constructed Margulis-systems for diffeomorphisms satisfying a condition called \emph{homogeneous exponential growth}. This property is the one shown for any $f\in \mathcal A^\infty_{1,1,1}(\mathbb T^3)$ in \cite[Proposition 3.1]{Humu}. While the two constructions are different, they both rely on the fact that the strong-unstable foliation is $1$-dimensional, i.e., that $d_u=1$. Moreover, establishing the quasi-invariance by $cs$-holonomy of the $u$-system was the main difficulty of \cite{Humu}. In that paper, we constructed from the spectral theory the $cs$-Margulis system as the restriction of the resonant state. We then used that $E_u$ was absolutely continuous to deduce that the $cs$-system was holonomy invariant. Finally, we deduced the holonomy invariance of the $u$-system using a Fubini-type argument.

In the present paper, we give a much more direct construction of the $u$-system which has several advantages as we explain below. In Proposition \ref{prop:umarg}, we construct a holonomy-invariant $u$-system directly, using the conjugacy to the linear model. The new idea is to work on the universal cover and use that any $f\in \mathcal A^\infty_{A}(\mathbb T^d)$ is \emph{dynamically coherent}, has a \emph{global $cs/u$-product structure} and is \emph{center-fibered}, see \cite{FPS}.  In particular, the argument does \emph{not} rely on quasi-invariance of the strong-unstable foliation and hence works when $d_u\geq 2$. Theorems \ref{main theo} and \ref{theores} extend with the same proofs to any $f\in \mathcal A^\infty_{d_s,d_c,d_u}(\mathbb T^d)$ having the three properties stated above.

The analytical argument of \cite[Section 4]{Humu} used the dynamical zeta function and some of the proofs there used that we worked on the $3$-dimensional torus. In \S \ref{sec:resona} we streamline this argument to make it work in any dimension. Finally, the $cs$-Margulis system could be constructed from the resonant state using the same argument as in \cite[\S 5.1]{Humu}. Since we already have a holonomy-invariant Margulis $u$-system $(\mu^u_x)_{x\in \mathbb T^d}$, we present instead in \S \ref{sec:cs} a more direct and more elementary construction of the $cs$-system directly from $(\mu^u_x)_{x\in \mathbb T^d}$ and $\mu_{J^c_f}$.
\subsection{Organisation of the paper} In \S \ref{sec:diffeo}, we define $ \mathcal A^\infty_{A}(\mathbb T^d)$ for $A\in \mathrm{SL}(d,\mathbb Z)\cap \mathcal A^\infty_{d_s,d_c,d_u}(\mathbb T^d)$ and recall the definitions of {dynamical coherence}, of having a {global $cs/u$-product structure} and being {center-fibered} which will be used in the proofs of the main theorems. In \S \ref{sec:system}, we recall the definition of systems of measures and prove the inequality $P(J^c_f)\leq h^u_{\mathrm{top}}(f)$ in Lemma \ref{lemm:ineq}.

In \S \ref{sec:Ruelle}, we study the Pollicott--Ruelle spectrum of $f$ acting on $(d_s+d_c)$-forms and show, using a trace formula, that the first resonance is simple and located at $e^{P(J^c_f)}$. 

In \S \ref{sec:usystem}, we construct a honolomy-invariant $u$-system with dilatation factor $\lambda^u(A)$, see Proposition \ref{prop:umarg}. We then show that the $u$-system defines a coresonant state, see Lemma \ref{lemmWF}. In particular, this shows that $\lambda^u(A)\leq e^{P(J^c_f)}$ from which we deduce that $\mu_{J^c_f}$ is a $u$-MME, see Corollary \ref{corollary}.

In \S \ref{sec:cs}, we construct the $cs$-Margulis system directly from the $u$-system and the $u$-MME, see Proposition \ref{prop:cs}.

Finally in \S \ref{sec:last}, we prove Corollaries \ref{cor1}, \ref{cor2} and \ref{cor3}.
\\
\textbf{Acknowledgements.} We would like to thank Clark Butler, Sylvain Crovisier, Colin Guillarmou, Thibault Lefeuvre, Rafael Potrie and Amie Wilkinson for discussions related to this project. More precisely, we would like to thank Rafael Potrie for pointing \cite{FPS} to our attention and answering some questions on the quasi-isometry of foliations.

We acknowledge that the idea of the construction in Proposition \ref{prop:cs} as well as some of the examples given in \S \ref{sec:last} were suggested during discussions with ChatGPT 5.6 Sol.  The rest of the paper has no AI input and all the proofs (including that of Proposition \ref{prop:cs}) were written by the author. The author bears full responsibility for their correctness.

This research was supported by the European Research Council (ERC) under
the European Union’s Horizon 2020 research and innovation programme (Grant agreement no. 101162990 — ADG).
\section{Preliminaries}
\subsection{Anosov diffeomorphisms with a partially hyperbolic splitting}
\label{sec:diffeo}
\subsubsection{Definition}
Let $M$ be a closed and smooth manifold and let $f\in \mathrm{Diff}^\infty(M)$ be a smooth diffeomorphism on $M$. Suppose that there is a continuous and $df$-invariant splitting of the tangent bundle $TM=E_s\oplus E_c\oplus E_u$ a norm $\|.\|$ and $\lambda,\mu>1$, such that for any $x\in M$ and any unit vector $v_\bullet(x)\in E_\bullet(x)$, for $\bullet=s,c,u$, we have
\begin{enumerate}
\item $\mu \|df_x(v_s)\|<\|df_x(v_c)\|<\mu^{-1}\|df_x(v_u)\|,$
\item $\|df_x(v_s)\|<\lambda^{-1}<1<\lambda < \|df_x(v_c)\|\leq \|df_x(v_u)\|.$
\end{enumerate}
The first point means that $f$ is a (pointwise) partially hyperbolic diffeomorphism while the second one means that $f$ is \emph{Anosov} with splitting $TM=E_s\oplus E_{cu}$ where $E_{cu}:=E_c\oplus E_u$. We will call $E_s$ (resp. $E_c$) the stable (resp. center) bundle of $f$. The bundle $E_{u}$ will be referred to as the strong-unstable bundle of $f$ while the bundle $E_{cu}$ will be referred to as the unstable bundle of $f$. Let $d_\bullet:=\mathrm{dim}(E_\bullet)$ for $\bullet=s,c,u$ and let $\mathcal A^{\infty}_{d_s,d_c,d_u}(M)$ denote the set of such diffeomorphisms. This set is easily seen to be $C^1$-open and examples are provided by linear Anosov automorphisms $A\in \mathrm{SL}(d,\mathbb Z)$ with a corresponding splitting of its (generalized) eigenspaces.
\subsubsection{Invariant manifolds}
\label{sec:mani} Let $f\in \mathcal A^{\infty}_{d_s,d_c,d_u}(M)$. Since $f$ is Anosov, $E_s$ and $E_{cu}$ integrate to $f$-invariant foliations $\mathcal W^s$ and $\mathcal W^{cu}$ called respectively the stable and unstable foliations of $f$, see \cite[Theorem 6.2.3]{HK}. Moreover, since $E_u$ is the strong-unstable bundle of a partially hyperbolic diffeomorphism, it integrates to a $f$-invariant foliation $\mathcal W^u$ called the strong-unstable foliation of $f$, see \cite[Theorem 4.1]{Pe}.

The existence of a $f$-invariant foliation $\mathcal W^{cs}$ tangent to the center-stable bundle $E_{cs}:=E_c\oplus E_s$ is a subtler problem. If such a foliation exists, we say that $f$ is \emph{dynamically coherent}. Note that in this case, the collection of leaves $\mathcal W^c(x):=\mathcal W^{cs}(x)\cap \mathcal W^{cu}(x)$ for $x\in \mathbb T^d$ defines the center foliation which integrates the center bundle $E_c$. Dynamical coherence for partially hyperbolic diffeomorphisms on tori is known in the following cases:
\begin{itemize}
\item if $M=\mathbb T^3$, see \cite{BBI,Po};
\item if $M=\mathbb T^d$ and $f$ is in the connected component (in the space of partially hyperbolic diffeomorphisms) of a linear automorphism $A\in \mathrm{SL}(d,\mathbb Z)$, see \cite[Main theorem]{FPS}.
\end{itemize}
For any $\bullet=s,c,u,cs,cu$, any $x\in M$ and any $\delta>0$, let $\mathcal W^\bullet(x,\delta):=\{y\in \mathcal W^\bullet(x) \mid d^\bullet(x,y)<\delta\}$, where $d^\bullet$ is the induced distance on the leaf.
\subsubsection{Local product structure}
\label{sec:product} Suppose that $f$ is dynamically coherent.
By transversality of the foliations $\mathcal W^{cs}$ and $\mathcal W^u$, there is an $\eps_0>0$ and a $C>0$ such that 
\begin{equation}
\label{eq:Bowen}\forall x,y\in \mathbb T^d, \ d(x,y)<\eps<\eps_0  \ \Rightarrow \ \mathcal W^{cs}(x,C\eps)\cap \mathcal W^u(y,C\eps)=\{[x,y]_{cs,u}\},
\end{equation}
where $[x,y]_{cs,u}$ is the \emph{Bowen bracket} of $x$ and $y$.
\begin{definition}
\label{defi:Rect}
A closed set $R\subset \mathbb T^d$ is called a \emph{rectangle} if for every $x,y\in R$, the Bowen bracket $[x,y]_{cs,u}$ exists and $[x,y]_{cs,u}\in R$.
\end{definition}
Rectangles exist because of the local product structure, see \cite[Eq. $(2.2)$]{CPZ20}.
We now define the center-stable and unstable holonomies. Let $R$ be a rectangle and $x\in  R$. We will denote by 
$\mathcal W^{\bullet}_R(x),$ for $ \bullet =cs,u,$
the connected component of $\mathcal W^{\bullet}(x)\cap R$ containing $x$.
For $x,y\in R$, the \emph{strong-unstable holonomy} is 
\begin{equation}
\label{eq:holu}
\mathrm{Hol}^u_{x,y}: \mathcal W^{cs}_R(x)\to \mathcal W^{cs}_R(y),\quad \mathrm{Hol}^u_{x,y}(z):=[y,z]_{cs,u}\in \mathcal W^{cs}(y)\cap \mathcal W^u(z).
\end{equation}
The \emph{center-stable holonomy} is 
\begin{equation}
\label{eq:holcs}
\mathrm{Hol}^{cs}_{x,y}: \mathcal W^{u}_R(x)\to \mathcal W^{u}_R(y),\quad \mathrm{Hol}^{cs}_{x,y}(z):=[z,y]_{cs,u}\in \mathcal W^{cs}(z)\cap \mathcal W^u(y).
\end{equation}
\subsubsection{Lift to the universal cover}
\label{sec:univ}
Let $A\in \mathrm{SL}(d,\mathbb Z)$ be a linear Anosov automorphism of $\mathbb T^d$, which admits a dominated splitting of the form $\mathbb R^d=E_s^A\oplus E_c^A\oplus E_u^A$, that is $A\in \mathcal A^\infty_{d_s,d_c,d_u}(\mathbb T^d)$, where $d_\bullet:=\mathrm{dim}(E_\bullet^A)$. Let $f\in \mathcal A^\infty_{d_s,d_c,d_u}(\mathbb T^d)$ be isotopic to $A$. 

Note that $A$ coincides with the action on homology of $f$. Recall that there exists a Hölder continuous homeomorphism $h:\mathbb T^d\to \mathbb T^d$  such that
$ h\circ f=A\circ h, $
see for instance \cite[Theorem 18.6.1]{HK}.  Let $H$, $F$ and $A$ denote a lift to $\mathbb R^d$ of $h,f$ and $A$ respectively. For a foliation $\mathcal N$ of $ \mathbb T^d$, we denote by $\widetilde{\mathcal N}$ the lifted foliation of $\mathbb R^d$. Then one has: 
\begin{itemize}
\item the conjugacy lifts to $\mathbb R^d$, that is,
\begin{equation}
\label{eq:conj2}
H\circ F=A\circ H; \quad \forall \tilde x\in \mathbb R^d,\forall \gamma\in \mathbb Z^d, \quad H(\tilde x+\gamma)=H(\tilde x)+\gamma;
\end{equation}
\item  for any $\bullet=s,cu$, one has
\begin{equation}
\label{eq:foliations}
\forall \tilde x\in \mathbb R^d,\quad H\big(\widetilde{ \mathcal W_f^\bullet}(\tilde x)\big)=\widetilde{\mathcal W^\bullet_A}(H(\tilde x));
\end{equation}
\end{itemize}
For a dynamically coherent $f\in \mathcal A^\infty_{d_s,d_c,d_u}(\mathbb T^d)$, we say that $f$ is \emph{center-fibered} if 
\begin{equation}
\label{eq:cfibered}
\forall \tilde x\in \mathbb R^d,\quad H\big(\widetilde{ \mathcal W_f^{c}}(\tilde x)\big)=\widetilde{\mathcal W^c_A}(H(\tilde x)),\quad H\big(\widetilde{ \mathcal W_f^{cs}}(\tilde x)\big)=\widetilde{\mathcal W^{cs}_A}(H(\tilde x)).
\end{equation}
\subsubsection{Global product structure}
\label{sec:globalproduct}
 We say that $f$ has a \emph{global $cs/u$-product structure} if
\begin{equation}
\label{eq:globalproduct}
\forall \tilde x,\tilde y\in \mathbb R^d,\quad  \widetilde{ \mathcal W_f^{cs}}(\tilde x) \text{ and } \widetilde{ \mathcal W_f^{u}}(\tilde y) \text{ intersect at a single point.}
\end{equation}
In this case, the intersection point is denoted by $[\tilde x,\tilde y]_{cs,u}$ and is called the \emph{Bowen bracket} of $x$ and $y$. We can define as in Definition \ref{defi:Rect} the notion of a rectangle $R\subset \mathbb R^d$. Moreover, for any $\tilde x,\tilde y\in \mathbb R^d$, we can define the strong-unstable (resp. center-stable) holonomy $\mathrm{Hol}^u_{\tilde x,\tilde y}: \widetilde{\mathcal W^{cs}}(\tilde x)\to  \widetilde{\mathcal W^{cs}}(\tilde y)$ (resp. $\mathrm{Hol}^{cs}_{\tilde x,\tilde y}: \widetilde{\mathcal W^{u}}(\tilde x)\to  \widetilde{\mathcal W^{u}}(\tilde y)$) by \eqref{eq:holu} (resp. \eqref{eq:holcs}). 

Denote by $\mathcal A^\infty_A(\mathbb T^d)$ the set of $f\in \mathcal A^\infty_{d_s,d_c,d_u}(\mathbb T^d)$ which are in the connected component of $A$ in the space of partially hyperbolic diffeomorphisms. We will use the following result due to Fisher, Potrie and Sambarino \cite[Theorem A, Theorem 4.1]{FPS}.
\begin{theorem}
\label{theo:FPS}
Let $A\in \mathrm{SL}(d,\mathbb Z)\cap \mathcal A^\infty_{d_s,d_c,d_u}(\mathbb T^d)$ and let  $f\in \mathcal A^\infty_A(\mathbb T^d)$. Then
\begin{itemize}
\item $f$ is dynamically coherent;
\item $f$ is center-fibered;
\item $f$ has global $cs/u$-product structure.
\end{itemize}
\end{theorem}
We also note that the conclusions of the theorem hold for any $f\in \mathcal A^\infty_{1,1,1}(\mathbb T^3)$ by results of Potrie \cite{Po}.

\subsection{Measures on the leaves}
\label{sec:system}
\subsubsection{System of measures}
Let $f\in \mathcal A^\infty_{d_s,d_c,d_u}(M)$ be dynamically coherent. In this paper, we construct \emph{systems of measures} supported on center-stable and strong-unstable leaves. We start with a definition.
\begin{definition}[System of measures]
\label{def}
For $\bullet=u,cs$, a system of  $\bullet$-measures $\mu^\bullet$ is a family $\mu^\bullet=(\mu^\bullet_x)_{x\in M}$ such that:
\begin{enumerate}
\item for any $x\in M$, $\mu^\bullet_x$ is a Borelian measure on $\mathcal W^\bullet(x)$ which is finite on compact subsets of $\mathcal W^\bullet(x)$;
\item for any $x,y\in M$ such that  $\mathcal W^\bullet(x)= \mathcal W^\bullet(y)$, one has $\mu^\bullet_x=\mu^\bullet_y$.
\end{enumerate}
We say that $\mu^\bullet$ is a measurable $($resp. continuous$)$ system of measures if $M$ is covered by foliation charts $B$ such that for any $\phi \in C_c^{\infty}(B,\mathbb R)$, the mapping $x\mapsto \mu^\bullet_x(\phi_{|\mathcal W^\bullet_{B}}(x))$ is measurable $($resp. continuous$)$. Here, $\mathcal W^{\bullet}_{B}(x)$ denotes the connected component of $\mathcal W^\bullet(x)\cap B$ containing $x$.
\end{definition}
We will say that a $u$-system of measures $\mu$ is \emph{$\mathrm{Hol}^{cs}$-invariant} if for any rectangle $R\subset M$ and any $p,q\in R$, we have $\mu_q=(\mathrm{Hol}^{cs}_{p,q})_*\mu_p$. 
\subsubsection{Conditional measures}
Following the initial idea of Margulis, the system of measures we will construct will be equal to the \emph{conditional measures} of the $u$-MME $\mu_{J^c_f}$. We do not recall the general definition of conditional measures and refer to \cite[\S 2.2.1]{Humu} for a more detailed discussion. We will only use that for any rectangle $R\subset M$, any probability measure $\mu$ on $R$ and $q\in R$, if we let 
\begin{equation}
\label{eq:tildemu}
\forall E\subset \mathcal W^{cs}_R(q),\quad \tilde \mu_q(E):=\mu(\cup_{y\in E}\mathcal W^{u}_R(y)),
\end{equation}
then there exists a system of $cs$-conditional measures  $(\mu^{u}_x)_{x\in R}$ such that 
\begin{itemize}
\item for any $x\in R$, $\mu^{u}_x$ is a probability measure on $\mathcal W^{u}_R(x)$,
\item if $\mathcal W^{u}_R(x)=\mathcal W^{u}_R(y)$, then $\mu^{u}_x=\mu^{u}_y$,
\item one has
\begin{equation}
\label{eq:CPZ}
\forall \phi\in L^1(R,\mu),\quad \mu(\phi)=\int_{\mathcal W^{cs}_R(q)}\int_{\mathcal W^{u}_R(y)}\phi(z)d\mu^{u}_y(z)d\tilde \mu_q(y),
\end{equation}
see for instance \cite[Equation $(2.6)$]{CPZ20}. 
\end{itemize}
One defines the \emph{$cs$-conditionals} $(\mu^{cs}_x)_{x\in R}$ of $\mu$ by replacing $u$ by $cs$ in the previous definition. Although this is not explicit in the notations, the conditionals  depend on the rectangle $R$. Note however that the ambiguity only consists in a normalizing constant, see \cite[Lemma 2.10]{CPZ20}.

\subsection{Unstable entropy}
\label{sec:entropy}
 Let $f\in \mathcal A^{\infty}_{d_s,d_c,d_u}(M)$ where  $M$ is a smooth and closed manifold. We will study  two notions of partial entropies associated to the strong-unstable bundle $E_u$ of $f$. The first one is the \emph{topological unstable entropy}. It was introduced by Saghin and Xia \cite{SaXi}, see also the work of Hu, Hua and Wu \cite{HuHuWu}. It is the exponential growth rate of the volume of strong-unstable balls.
\begin{equation}
\label{eq:hutop1}
 \forall \delta>0, \quad h^u_{\mathrm{top}}(f):=\sup_{x\in M}\limsup_{n\to+\infty}\frac 1n \log \mathrm{Vol}\big(f^n(\mathcal W^u(x,\delta))\big).
\end{equation}
The second notion of entropy is the \emph{metric unstable entropy} associated to $f$ and a $f$-invariant probability measure $\mu\in \mathcal P_f(M)$. We will denote it by $h_{\mu}^u(f)$ and we refer to \cite{HuHuWu} for a detailed discussion of its definition and properties. What will be used in this paper is the fact that the two notions of entropy are linked via the {variational principle} \eqref{eq:varia111}, see \cite[Theorem D]{HuHuWu}. Recall that the set of $u$-MMEs will be denoted by $\mathcal{M}^u(f)$.
\subsection{Center Jacobian}
In this subsection, we define the \emph{center Jacobian}. Let
\begin{equation}
\label{eq:Jc}
J^c_f\in C^{\alpha}(M,\mathbb R),\ \forall x\in M, \quad J_f^c(x):=-\log |\det(df_x|_{E_c(x)})|,
\end{equation}
for some $\alpha>0$.
By the thermodynamic formalism (see \cite[Theorem 20.3.7]{HK}), there is a unique \emph{equilibrium state} $\mu_{J^c_f}$ associated to $J_f^c$. It is the unique invariant probability measure such that
\begin{equation}
\label{eq:varia}
h_{\mu_{J^c_f}}(f)+\int_{M}J^c_f(p)d\mu_{J^c_f}(p)=P(J^c_f):=\sup_{\mu \in \mathcal P_f(M)}h_{\mu}(f)+\int_{M}J^c_f(p)d\mu(p),
\end{equation}
where $P(J^c_f)$ is the pressure of the center Jacobian and $h_{\mu}(f)$ is the metric entropy of the map $f$ with respect to $\mu$. We  show an inequality between $h^u_{\mathrm{top}}(f)$ and $P(J^c_f)$.
\begin{lemma}
\label{lemm:ineq}
Let $f\in\mathcal A^{\infty}_{d_s,d_c,d_u}(M)$. Then one has
\begin{equation}
\label{eq:ineq}
h^u_{\mathrm{top}}(f)\geq P(J^c_f).
\end{equation}
If the above inequality  is an equality, then $\mu_{J^c_f}\in \mathcal M^u(f).$
\end{lemma}
\begin{proof}
Since $f$ is Anosov, the equilibrium state $\mu_{J^c_f}$ is ergodic. 
We use the Ledrappier-Young inequality \cite{LY} (see also \cite[Corollary A.1]{HuHuWu}) which states that 
\begin{equation}
\label{eq:HuHuWu}
h_{\mu_{J^c_f}}(f)\leq h^u_{\mu_{J^c_f}}(f)+\sum_{i=1}^rm_i\lambda_{c,i}(f,\mu_{J^c_f}), 
\end{equation}
where  $\lambda_{c,1}(f,\mu_{J^c_f})\geq \lambda_{c,2}(f,\mu_{J^c_f})\geq \ldots \geq \lambda_{c,r}(f,\mu_{J^c_f})>0$ are the central Lyapunov exponents of $\mu_{J^c_f}$ (i.e. the distinct Lyapunov exponents of the cocycle $df|_{E_c}$ with respect to $\mu_{J^c_f}$) and $m_i$ are their multiplicities. Since $E_c$ is uniformly expanded, all the Lyapunov exponents are positive.

By Oseledet's theorem applied to $df|_{E_c}$, for  $\mu_{J^c_f}$-a.e. $x\in M$, we have
$$\lim_{n\to+\infty}\frac 1n\log |\det(df^n_x|_{E_c(x)})|= \sum_{i=1}^rm_i\lambda_{c,i}(f, \mu_{J^c_f}).$$
Recall that we have 
$J^c_f(x)=-\log |\mathrm{det}(df_x|_{E_c(x)})|. $ The chain rule then implies, for any $n\geq 1$,
$$\frac 1n \sum_{k=0}^{n-1}J^c_f(f^kx)=  -\frac 1n\log |\det(df^n_x|_{E_c(x)})|.$$
Now, the measure $\mu_{J^c_f}$ is ergodic so by the ergodic theorem, for $\mu_{J^c_f}$-a.e. $x\in M$, 
$$ \int_{M}J^c_f d\mu_{J^c_f}=-\sum_{i=1}^rm_i\lambda_{c,i}(f, \mu_{J^c_f})=\lim_{n\to+\infty}\frac 1n \sum_{k=0}^{n-1}J^c_f(f^kx).$$
Plugging this into \eqref{eq:HuHuWu} and using \eqref{eq:varia} yields
\begin{equation}
\label{eq:uvaria}
h^u_{\mu_{J^c_f}}(f)\geq h_{\mu_{J^c_f}}(f)-\sum_{i=1}^rm_i\lambda_{c,i}(f, \mu_{J^c_f})=h_{\mu_{J^c_f}}(f)+\int_{M}J^c_fd\mu_{J^c_f}=P(J^c_f).
\end{equation}
Using $h^u_{\mu_{J^c_f}}(f)\leq h^u_{\mathrm{top}}(f)$ gives \eqref{eq:ineq}. Moreover, equality implies $\mu_{J^c_f}\in \mathcal M^u(f).$
\end{proof}

\section{The first Pollicott-Ruelle resonance on $(d_s+d_c)$-forms}
\label{sec:Ruelle}
In this section, we consider a smooth closed manifold $M$ and $f\in \mathcal A_{d_s,d_c,d_u}^\infty(M)$. The goal is to study the \emph{Pollicott-Ruelle spectrum} of $f$ acting on $(d_s+d_c)$-forms. 
\subsection{The Pollicott-Ruelle resonances} 
In this subsection, we recall some facts about the functional approach for Anosov diffeomorphisms. For a detailed introduction to the subject, we refer to \cite{Ba} and \cite[Section 4.1]{Humu}. For $0\leq k \leq d$, define the space of smooth $k$-forms
$\Omega_k= C^{\infty}(M; \Lambda^k T^* M). $
The diffeomorphism $f$ acts on $\Omega_k$ by pushforward:
\begin{equation}
\label{eq:Lie}
f_{*,k}: \Omega_k\to \Omega_k, \quad f _{*,k} \omega:=f_*\omega.
\end{equation}
For an Anosov diffeomorphism, one can associate to $f_{*,k}$ a discrete spectrum $\mathrm{Res}_k(f)$, the \emph{Pollicott-Ruelle resonances}, by making it act on specially designed \emph{anisotropic spaces}, see for instance \cite{BKL,BT,GouLiv,Fau08}. More precisely,  the resolvent $R_k(\lambda):=(f_{*,k}-\lambda)^{-1}:L^2(M;\Lambda^kT^*M)\to L^2(M;\Lambda^kT^*M)$, defined for $\lambda \in \mathbb C$ with $|\lambda|\gg 1$ admits a meromorphic extension $R_k(\lambda)$ to $\mathbb C\setminus \{0\}$,
$$R_k(\lambda):\Omega_k\to \mathcal D'_k(M). $$ 
The poles of the extension are intrinsic and are called the \emph{Pollicott-Ruelle resonances $\mathrm{Res}_k(f)$ of $f_{*,k}$} (or the Pollicott--Ruelle resonances of $f$ acting on $k$-forms). For any $\lambda_0\in \mathrm{Res}_k(f)$, the spectral projector at $\lambda_0$ is given by
$$ \Pi_k^{\lambda_0}=-\frac{1}{2i\pi}\int_{\gamma}R_k(z)dz,$$
where $\gamma$ is a small loop around $\lambda_0$. The projector has finite rank. The (algebraic) multiplicity of $\lambda_0$ is given by the rank of the spectral projector:
\begin{equation}
\label{eq:mult}
\forall \lambda_0\in \mathrm{Res}_k(f), \quad m_k(\lambda_0):=\mathrm{rank}(\Pi_k^{\lambda_0})\in \mathbb N.
\end{equation}
Near a resonance $\lambda_0\in \mathrm{Res}_k(f)$, the resolvent admits a Laurent expansion of the form
\begin{equation}
\label{eq:Laurent}R_k(\lambda)=R_k^H(\lambda)-\sum_{j=1}^{N(\lambda_0)}\frac{(f_{*,k} -\lambda_0)^{j-1}\Pi^{\lambda_0}_k}{(\lambda-\lambda_0)^j}, 
\end{equation}
where $R_k^H(\lambda)$ is the holomorphic part near $\lambda_0$. The \emph{generalized resonant states} are the elements in the range of the spectral projector. Define a dual decomposition to the Anosov decomposition:
\begin{equation}
\label{eq:dual}
T^*M=E_s^*\oplus E_{cu}^*, \quad E_\bullet^*(E_\bullet)=0, \ \bullet=s,cu.
\end{equation}
We have the following characterization of generalized resonant states:
$$
\mathrm{Res}_{k,\lambda_0,\infty}(f):=\Pi^{\lambda_0}_k(\Omega_k)=\{u\in \mathcal D'_k(M)\mid (f_{*,k}-\lambda_0)^{N(\lambda_0)}u=0,\ \mathrm{WF}(u)\subset E_{cu}^*\},
$$
where $\mathrm{WF}(u)$ denotes the wavefront set of a current $u$, see \cite[Chapter VIII]{Hor}.
The \emph{resonant states} are defined as 
\begin{equation}
\label{eq:resonant}
\mathrm{Res}_{k,\lambda_0}(f):=\{u\in \mathcal D'_k(M)\mid (f_{*,k}-\lambda_0)u=0,\ \mathrm{WF}(u)\subset E_{cu}^*\}. \end{equation}
We will say that $f$ has no Jordan block at $\lambda_0$ if $N(\lambda_0)=1$, that is, if $\mathrm{Res}_{k,\lambda_0}(f)=\mathrm{Res}_{k,\lambda_0,\infty}(f)$. Note that applying the previous results to $f^{-1}$ yields notions of (generalized) co-resonant states, see \eqref{eq:cores}.

For $m\geq 1$, let 
$
\mathcal F_m(f):=\{x\in M\mid f^m(x)=x\}.
$
We recall the definition of the \emph{dynamical determinants}. For $k\in \{ 0,1,\ldots, d\}$, let
\begin{equation}
\label{eq:D_f,k}
\forall w\in \mathbb C, |w|\ll 1, \quad \mathcal D_{f,k}(w):=\exp\left( -\sum_{m=1}^{+\infty} \frac {w^m}m\sum_{x\in \mathcal F_m(f)}\frac{\tr(\Lambda^kdf^{-m}(x))}{|\mathrm{det}(\mathrm{Id}-df^{-m}(x))|}\right).
\end{equation}
By \cite[Theorem 6.2]{Ba},  $\mathcal D_{f,k}$ admits a \emph{holomorphic} extension to $\mathbb C$. Moreover,
\begin{equation}
\label{eq:zero}
\forall w\in \mathbb C^*, \quad \mathcal D_{f,k}(w)=0 \ \iff \ w^{-1}\in \mathrm{Res}_k(f),
\end{equation}
 that is, $w$ is a zero of $ \mathcal D_{f,k}$ if and only if  $w^{-1}$ is a Pollicott-Ruelle resonance of $f$ for the action on $k$-forms. Moreover, the multiplicity of the zero coincides with the algebraic multiplicity $m_k(w^{-1})$ of $w^{-1}$.

\subsection{Computing the first resonance}
\label{sec:resona}
Equation \eqref{eq:zero} allows us to understand the spectral theory, that is, the Pollicott-Ruelle spectrum of $f$ by studying weighted sums over periodic points. Recall that $\lambda\in \mathrm{Res}_k(f)\cap \mathbb R_+$ is the first resonance for the action on $k$-forms if $\mathrm{Res}_k(f)\subset \{z\in \mathbb C,|z|\leq \lambda\}.$
\begin{proposition}[First Ruelle resonance]
\label{prop1st}
Let $f\in\mathcal A_{d_s,d_c,d_u}^\infty(M)$ be topologically mixing and satisfying Assumption \ref{assump}. Then $e^{P(J^c_f)}$ is the first Pollicott-Ruelle resonance for the action on $(d_c+d_s)$-forms.
\begin{equation}
\label{eq:equality}
\sup\{|z|, z\in \mathrm{Res}_{d_s+d_c}(f)\}=e^{P(J^c_f)}.
\end{equation}
Moreover, $e^{P(J^c_f)}$ has no Jordan block and has multiplicity equal to one and is the only resonance on the critical circle $\{z\in \mathbb C, \ |z|=e^{P(J^c_f)}\}$. Finally, the trace of the spectral projector of $f$ at $z=e^{P(J^c_f)}$ is equal to $\mu_{J^c_f}$, the equilibrium state associated to $J^c_f$.
\end{proposition}

 \begin{proof} Using the dominated splitting $E_{cu}=E_c\oplus E_u$ and \eqref{eq:zero}, we will show that $\mathcal D_{f,d_s+d_c}(z)$ is essentially given (on the critical circle of convergence) by the exponentials of Birkhoff sums of $J^c_f$ over periodic points. 
 
 More precisely,  let $m\geq 1$ and let $x\in \mathcal F_m(f)$. Let us understand the spectrum of the linear map $df^{-m}(x)$. Since $f\in\mathcal A_{d_s,d_c,d_u}^{\infty}(M)$, we see that $df^{-m}(x)$ decomposes as a block matrix with blocks $df^{-m}(x)|_{E_\bullet(x)}$ for $\bullet=s,c,u$. Moreover, there exist constants $C,\nu>0$ such that for any  $x\in M$ and $m\geq 1$,
 \begin{equation}
 \label{eq:eigen fn}
 \begin{split}
 \min_{\mu \in \mathrm{Spec}(df^{-m}(x)|_{E_s(x)})}|\mu|&\geq Ce^{\nu m}\geq 1 \geq C e^{-\nu m}\geq \max_{\mu \in \mathrm{Spec}(df^{-m}(x)|_{E_c(x)})}|\mu|
 \\&\geq C e^{-\nu m} \min_{\mu \in \mathrm{Spec}(df^{-m}(x)|_{E_c(x)})}|\mu|\geq \max_{\mu \in \mathrm{Spec}(df^{-m}(x)|_{E_u(x)})}|\mu|.
 \end{split}
 \end{equation}
 If we denote by $\rho^m_\bullet(x):=|\mathrm{det}(df^{-m}(x)|_{E_\bullet(x)})|$ for $ \bullet =s,c,u$, we deduce, using that $f$ satisfies Assumption \ref{assump}, when $m\to+\infty$, 
 \begin{equation}
 \label{eq:equiv}
 \begin{split}
 &|\mathrm{det}(\mathrm{Id}-df^{-m}(x))|=\rho^m_s(x)\big(1+O(e^{-\nu m})\big),
 \\ &\tr(\Lambda^{d_s+d_c}df^{-m}(x))=\rho^m_s(x)\rho^m_c(x)\big(1+O(e^{-\nu m})\big),
 \end{split}
 \end{equation}
  uniformly in $x\in M$. Plugging \eqref{eq:equiv} into \eqref{eq:D_f,k} gives
  \begin{equation} 
  \label{eq:unif}\sum_{x\in \mathcal F_m(f)}\frac{\tr(\Lambda^{d_s+d_c}df^{-m}(x))}{|\mathrm{det}(\mathrm{Id}-df^{-m}(x))|}=\sum_{x\in \mathcal F_m(f)}\rho_c^m(x)\big(1+O(e^{-\nu m})\big)=\left(\sum_{x\in \mathcal F_m(f)}\rho_c^m(x)\right)\big(1+O(e^{-\nu m})\big).
  \end{equation}
Now, using the chain rule, we obtain
$$ \forall m \geq 1,\ \forall x\in \mathcal F_m(f), \quad \rho^m_c(x)=|\mathrm{det}(df^m(x)|_{E_c(x)})|^{-1}=\exp\left( \sum_{k=0}^{m-1}J_f^c(f^kx) \right).$$
Define
$$\zeta_{f, J^c_f}(w):=\exp\left(-\sum_{m=1}^{+\infty} \frac{w^m}m\sum_{x\in \mathcal F_m(f)} \exp\left( \sum_{k=0}^{m-1}J_f^c(f^kx) \right)\right).$$
This is the Ruelle zeta function with potential $J^c_f$. Since $f$ is topologically mixing, by \cite[Theorem 1]{Ru}, there is a $R=R(J^c_f)>e^{-P(J^c_f)}$ such that $\zeta_{f, J^c_f}$ has no zero in $ \overline{B}(0,R)$ except for a simple zero at  $e^{-P(J^c_f)}$. 

We have $\mathcal D_{f,d_c+d_s}(w)=\zeta_{f, J^c_f}(w)e^{-G(w)}$, where
$$G(w):=\sum_{m=1}^{+\infty}\frac{w^m}m\sum_{x\in \mathcal F_m(f)}\left(\frac{\tr(\Lambda^{d_s+d_c}df^{-m}(x))}{|\mathrm{det}(\mathrm{Id}-df^{-m}(x))|}-\exp\left( \sum_{k=0}^{m-1}J_f^c(f^kx) \right)\right). $$
By \eqref{eq:unif} and the standard estimate 
$$\sum_{x\in \mathcal F_m(f)}\exp\left( \sum_{k=0}^{m-1}J_f^c(f^kx) \right)=O(e^{mP(J^c_f)}), $$
we see that $G(w)$ is holomorphic in a disk strictly larger than $\{z\in \mathbb C, |z|\leq e^{-P(J^c_f)}\}$ which means that $\mathcal D_{f,d_c+d_s}$ also has no zero in $ \overline{B}(0,R)$ except for a simple zero at  $e^{-P(J^c_f)}$. By \eqref{eq:zero}, this shows that $e^{P(J^c_f)}$ is the first resonance and that it is the only one on the critical circle.
 
 We now study the trace of the spectral projector at the first resonance. 
 Let $g\in C^{\infty}(M,\mathbb R_+)$ be a smooth and positive function. 
 The Guillemin trace formula (see for instance \cite[Proposition 6.3]{BaTs}) gives for any $m\geq 1$, 
$$\tr^{\flat}\big(gf_{*,d_s+d_c}^m\big)=\sum_{x\in \mathcal F_m(f)}\frac{g(x)\tr(\Lambda^{d_s+d_c}df^{-m}(x))}{|\mathrm{det}(\mathrm{Id}-df^{-m}(x))|}, $$
where $\tr^{\flat}$ denotes the \emph{flat trace} and $f_{*,d_s+d_c}$ denotes the pushforward on $(d_s+d_c)$-forms. For any $z\in \mathbb C$ such that $|z|\gg 1$, define $Z_{g,f}(z):=\tr^{\flat}\big(gf_{*,d_s+d_c}(z-f_{*,d_s+d_c})^{-1}\big)$. For $|z|\gg 1$,
\begin{equation}
\label{eq:Zfg}
Z_{g,f}(z)=\sum_{m=1}^{+\infty}z^{-m}\tr^{\flat}\big(gf_{*,{d_s+d_c}}^m\big)=\sum_{m=1}^{+\infty}z^{-m}\sum_{x\in \mathcal F_m(f)}\frac{g(x)\tr(\Lambda^{d_s+d_c}df^{-m}(x))}{|\mathrm{det}(\mathrm{Id}-df^{-m}(x))|}.
\end{equation}
%In the above computation, the commutation of the sum and trace is justified for $|z|\ll 1$ since the flat trace can be interpreted as the trace on the anisotropic spaces. 
Note that $Z_{g,f}$ admits a meromorphic extension to $\mathbb C\setminus \{0\}$ since the resolvent does, see \cite[Proposition 6.3]{BaTs}.
Using the first part of the proof, we see that $Z_{g,f}$ has a first simple pole at $z=e^{P(J^c_f)}$ and there is some $\eps>0$ such that any other pole has modulus smaller than $e^{P(J^c_f)}-\eps$. We write the spectral projector at the first resonance $e^{P(J^c_f)}$ as 
$ \Pi_0(e^{P(J^c_f)})=\nu(\cdot) \theta,$
where $\theta$ (resp.  $\nu$) is a resonant (resp. co-resonant state) which are normalized so that $(\nu,\theta)=1$, where $(.,.)$ denotes the distributional product. Recall from \eqref{eq:Laurent}, that near $z=e^{P(J^c_f)}$, 
\begin{align*}\tr^{\flat}\big(gf_{*,d_s+d_c}(z-f_{*,d_s+d_c})^{-1}\big)&=\tr^{\flat}\Big(gf_{*,d_s+d_c}\big(\tfrac{\Pi_0}{z-e^{P(J^c_f)}}+h(z)\big) \Big),
\end{align*}
for some holomorphic map $h$.
In particular, the residue at $e^{P(J^c_f)}$ is given by
\begin{align*}\mathrm{Res}_{e^{P(J^c_f)}}(Z_{g,f})&=\tr(gf_{*,d_s+d_c}\Pi_0(e^{P(J^c_f)}))=e^{P(J^c_f)}(\nu,g\times  \theta)=:e^{P(J^c_f)}(\nu\wedge \theta)(g).
\end{align*}
Next, we use Cauchy's formula on \eqref{eq:Zfg} to obtain, for some $\rho<e^{P(J^c_f)}$,
\begin{align*}\sum_{x\in \mathcal F_m(f)}\frac{g( x)\tr(\Lambda^{d_s+d_c}df^{-m}(x))}{|\mathrm{det}(\mathrm{Id}-df^{-m}(x))|}=&\mathrm{Res}_{e^{P(J^c_f)}}(Z_{g,f}(z)z^{m-1})+\frac 1{2i\pi}\int_{\rho \mathbb S^1}Z_{g,f}(z)z^{m-1}dz.
\end{align*}
Since $Z_{g,f}$ has (at most) a simple first pole $e^{P(J^c_f)}$, one has
$$\sum_{x\in \mathcal F_m(f)}\frac{g(x)\tr(\Lambda^{d_s+d_c}df^{-m}(x))}{|\mathrm{det}(\mathrm{Id}-df^{-m}(x))|} =e^{mP(J^c_f)}((\nu\wedge \theta)(g)+O(e^{-\eps m})),$$
for some $\eps>0$. Using \eqref{eq:equiv}, this gives
\begin{equation}
\label{eq:Bowen}
(\nu\wedge \theta)(g)=\lim_{m\to+\infty} e^{-mP(J^c_f)}\sum_{x\in \mathcal F_m(f)}\exp\left( \sum_{k=0}^{m-1}J_f^c(f^kx) \right)g(x).
\end{equation}
For $g\equiv 1$, we obtain
\begin{align*}\lim_{m\to+\infty} e^{-mP(J^c_f)}\sum_{x\in \mathcal F_m(f)}|\mathrm{det}(df^m(x)|_{E_c(x)})|^{-1}=1. 
\end{align*}
This means that \eqref{eq:Bowen} is Bowen's formula for the equilibrium state (see for instance  \cite[Theorem 20.3.7]{HK}):
$$(\nu\wedge \theta)(g)=\lim_{m\to+\infty} \frac{\sum_{x\in \mathcal F_m(f)}\exp\left( \sum_{k=0}^{m-1}J_f^c(f^kx) \right)g(x)}{\sum_{x\in \mathcal F_m(f)}\exp\left( \sum_{k=0}^{m-1}J_f^c(f^kx) \right)}=\mu_{J^c_f}(g), $$ which concludes the proof.
 \end{proof}

 \section{The $u$-Margulis system and the co-resonant state}
 \label{sec:usystem}
 In this section, we further suppose that $M=\mathbb T^d$ and that $f$ satisfies the conditions of Theorem \ref{main theo}.
  \subsection{Constructing the $u$-system}
The main result of this section is the following proposition which constructs a holonomy-invariant $u$-Margulis system.
\begin{proposition}
\label{prop:umarg}
Let $f\in\mathcal A^\infty_{d_s,d_c,d_u}(\mathbb T^d)$ satisfying Assumption \ref{assump}. Suppose that:
 \begin{itemize}
 \item either $d=3$;
\item or $f\in A^\infty_A(\mathbb T^d)$.
 \end{itemize}
Then there exists a $u$-system $(\mu^u_x)_{x\in \mathbb T^d}$ such that:
\begin{itemize}
\item for any $x\in \mathbb T^d$, $\mu^u_x$ has full support in $\mathcal W^u(x)$ and is atomless;
\item one has 
\begin{equation}
\label{eq:confu}
\forall  x\in \mathbb T^d,\quad f^*\mu^u_x=\lambda^u(A)\mu^u_{f^{-1}(x)};
\end{equation}
\item the system is $\mathrm{Hol}^{cs}$-invariant. For any rectangle $R\subset \mathbb T^d$ and any $p,q\in R$, we have $(\mathrm{Hol}^{cs}_{p,q})_*\mu^u_p=\mu^u_q.$
\end{itemize}
\end{proposition} 
\begin{proof}
As in \S \ref{sec:univ}, let $A$ be the action on homology of $f$. Recall that there exists a Hölder continuous homeomorphism $h:\mathbb T^d\to \mathbb T^d$  such that
$ h\circ f=A\circ h $.  Let $H$, $F$ and $A$ denote a lift to $\mathbb R^d$ of $h,f$ and $A$ respectively. Let $\pi^u_A:\mathbb R^d\to E_u^A$ denote the projection on $E_u^A$ parallel to $E_{cs}^A$. Let 
\begin{equation}
\label{eq:Q}
Q:=\pi^u_A\circ H: \mathbb R^d \to E_u^A.
\end{equation}
Since $f$ is center-fibered, see \eqref{eq:cfibered}, we have 
\begin{equation}
\label{eq:Qconst}
\forall \tilde x\in \mathbb R^d,\quad Q(\widetilde{\mathcal W^{cs}_f}(\tilde x))=\{Q(\tilde x)\},
\end{equation}
that is, the map $Q$ is constant on center-stable leaves. We make the following claim:
\begin{equation}
\label{eq:Qhomeo}
\forall \tilde x\in \mathbb R^d,\quad Q_{\tilde x}:=Q|_{\widetilde{\mathcal W^{u}_f}(\tilde x)}:\widetilde{\mathcal W^{u}_f}(\tilde x)\to E_u^A \text{ is a homeomorphism.}
\end{equation}
To prove the claim, we use that $f$ has a global $cs/u$-product structure, see \eqref{eq:globalproduct}. Suppose that there exist $\tilde y,\tilde z\in \widetilde{\mathcal W^{u}_f}(\tilde x)$ such that $Q(\tilde y)=Q(\tilde z)$, that is, such that $H(\tilde z)\in H(\widetilde{\mathcal W_f^{cs}}(\tilde y))$. Using \eqref{eq:cfibered} again and the fact that $H$ is a homeomorphism, this implies $\tilde z\in \widetilde{\mathcal W^{cs}_f}(\tilde y)$. In particular, we have 
$\tilde y,\tilde z\in \widetilde{\mathcal W^{cs}_f}(\tilde y)\cap \widetilde{\mathcal W^{u}_f}(\tilde x).$ Using \eqref{eq:globalproduct}, this gives $\tilde y=\tilde z$ which shows that $Q_{\tilde x}$ is injective.

For the surjectivity, let $v\in E^A_u$. Consider the intersection $(E_{cs}^A+v)\cap H(\widetilde{\mathcal W^{u}_f}(\tilde x)).$ We claim that this intersection is not empty. Indeed, using \eqref{eq:cfibered}, we have 
$$(E_{cs}^A+v)=\widetilde{\mathcal W^{cs}_A}(v)=H\big(\widetilde{\mathcal W^{cs}_f}(H^{-1}(v)) \big). $$
The claim then follows from the fact that $H$ is a homeomorphism and from \eqref{eq:globalproduct}. Let $w\in (E_{cs}^A+v)\cap H(\widetilde{\mathcal W^{u}_f}(\tilde x))$, then by construction, we have $Q_{\tilde x}(H^{-1}w)=v$ which shows that $Q_{\tilde x}$ is surjective.

Since $H$ is a homeomorphism, \eqref{eq:Q} implies that $Q_{\tilde x}$ is a continuous bijection from $\widetilde{\mathcal W^{u}_f}(\tilde x)$ to $E_u^A $. By  Brouwer's invariance of domain theorem, we deduce that $Q_{\tilde x}$ is a homeomorphism. We can now define our system of measures. We let
\begin{equation}
\label{eq:usystem}
\forall \tilde x\in \mathbb R^d, \quad \mu^u_{\tilde x}:=(Q^{-1}_{\tilde x})_*\mathrm{Leb}_{E^u_A},
\end{equation}
where $\mathrm{Leb}_{E^u_A}$ denotes the Lebesgue measure on the vector space $E^u_A$.
\\ \textbf{The $(\mu^u_{\tilde x})_{x\in \mathbb R^d}$ is finite on bounded sets.}
Since $H$ is continuous, \eqref{eq:conj2} implies that we have $C_H:=\sup_{\tilde x\in \mathbb R^d}d(H(\tilde x),\tilde x)<+\infty$. Now, if we let $\tilde y,\tilde z\in \widetilde{\mathcal W^u_f}(\tilde x)$, we have
\begin{align*}d_{\mathbb R^d}(Q_{\tilde x}(\tilde y),Q_{\tilde x}(\tilde z))&= d_{\mathbb R^d}(\pi_A^u(H(\tilde y)),\pi_A^u(H(\tilde z)))\leq \|\pi_A^u\| d_{\mathbb R^d}(H(\tilde y),H(\tilde z))
\\&\leq \|\pi_A^u\|(d_{\mathbb R^d}(\tilde y,\tilde z)+2C_H)\leq \|\pi_A^u\|(d_{u}(\tilde y,\tilde z)+2C_H),
\end{align*}
where $d_{u}$ denotes the distance on the leaf $\widetilde{\mathcal W^u_f}(\tilde x)$. This implies that any bounded set in $\widetilde{\mathcal W^u_f}(\tilde x)$ is sent by $Q_{\tilde x}$ to a bounded set in $E_u^A$ and this shows that $\mu^u_{\tilde x}$ is finite on bounded sets.
\\ \textbf{The system $(\mu^u_{\tilde x})_{x\in \mathbb R^d}$ descends to $\mathbb T^d$.} 
Let $\gamma \in \mathbb Z^d$ and let $\tilde x+\gamma$ be another lift of some $x\in \mathbb T^d$. Note that for any $\tilde z\in \mathbb R^d$, the $\mathbb Z^d$ equivariance of $H$ gives $Q(\tilde z+\gamma)=Q(\tilde z)+\pi^u_A(\gamma). $ Since the Lebesgue measure is translation-invariant, we have
\begin{equation}
\label{eq:Tgamma}
(T_\gamma)_*\mu^u_{\tilde x}=\mu^u_{\tilde x+\gamma}, \quad T_\gamma: \tilde z\mapsto \tilde z+\gamma. \end{equation}
This shows that the measure does not depend on the choice of lift for $x\in \mathbb T^d$. In particular,  the system $(\mu^u_{\tilde x})_{x\in \mathbb R^d}$ descends to a system of measures on $\mathbb T^d$ which we still denote by $(\mu^u_{\tilde x})_{x\in \mathbb T^d}$. All the properties claimed in Proposition \ref{prop:umarg} are local and can thus be proved for the lifted system $(\mu^u_{\tilde x})_{x\in \mathbb R^d}$ instead, which we will do in the rest of the proof.
\\ \textbf{The measure $\mu^u_{\tilde x}$ is atomless and has full support.} Note that since $Q_{\tilde x}$ is a homeomorphism, it is an open map. Let $\tilde Z\subset \widetilde{\mathcal W^{u}_f}(\tilde x)$ be a non-empty open subset. Then $ \mu^u_{\tilde x}(\tilde Z)=\mathrm{Leb}_{E^u_A}(Q_{\tilde x}(\tilde Z))>0$ since $\mathrm{Leb}_{E^u_A}$ has full support. This proves that $\mu^u_{\tilde x}$ has full support. Similarly, for any $\tilde y\in \widetilde{\mathcal W^{u}_f}(\tilde x)$, we have $ \mu^u_{\tilde x}(\{\tilde y\})=\mathrm{Leb}_{E^u_A}(Q_{\tilde x}(\{\tilde y\}))=0$ since $\mathrm{Leb}_{E^u_A}$ is atomless. This proves that $\mu^u_{\tilde x}$ is atomless.
\\ \textbf{The system $(\mu^u_{\tilde x})_{x\in \mathbb R^d}$ has the scaling property \eqref{eq:confu}.}
We start with the observation that $\pi^u_A\circ A=(A|_{E^A_u})\circ \pi^u_A$ since $A$ preserves its generalized eigenspaces. Let $\tilde x\in \mathbb R^d$ and let $\tilde Z\subset \widetilde{W^u_f}(F^{-1}\tilde x)$. We compute, using \eqref{eq:conj2} and \eqref{eq:Q}, 
\begin{equation}
\label{eq:CONF}
\begin{split}
F^*\mu^u_{\tilde x}(\tilde Z)&=\mu^u_{\tilde x}(F(\tilde Z))=\mathrm{Leb}_{E^A_u}(Q\circ F(\tilde Z))=\mathrm{Leb}_{E^A_u}(\pi^u_A\circ H\circ F(\tilde Z))
\\&=\mathrm{Leb}_{E^A_u}(\pi^u_A\circ A\circ H(\tilde Z))=\mathrm{Leb}_{E^A_u}((A|_{E^A_u})\circ \pi^u_A\circ H(\tilde Z))
\\&=\lambda^u(A)\mathrm{Leb}_{E^A_u}( \pi^u_A\circ H(\tilde Z))=\lambda^u(A)\mu^u_{F^{-1}\tilde x}(\tilde Z).
\end{split}
\end{equation}
This implies \eqref{eq:confu} for the induced system on $\mathbb T^d$.
\\ \textbf{The system $(\mu^u_{\tilde x})_{x\in \mathbb R^d}$ is ${\mathrm{Hol}^{cs}}$-invariant.} Let $\tilde x,\tilde y\in \mathbb R^d$ and consider the center-stable holonomy ${\mathrm{Hol}^{cs}_{\tilde x,\tilde y}}:\widetilde{W^u_f}(\tilde x)\to \widetilde{W^u_f}(\tilde y)$. Note that the holonomy is globally well defined since $f$  has global $cs/u$-product structure. Let $\tilde Z\subset \widetilde{W^u_f}(\tilde x)$, then since $f$ is center-fibered  (see \eqref{eq:cfibered}), we have $Q_{\tilde x}(\tilde Z)=Q_{\tilde y}({\mathrm{Hol}^{cs}_{\tilde x,\tilde y}}(\tilde Z)).$ In other words, we have shown that 
\begin{equation}
\label{eq:Holcsuniv}
\forall \tilde x,\tilde y\in \mathbb R^d,\quad ({\mathrm{Hol}^{cs}_{\tilde x,\tilde y}})_*\mu^u_{\tilde x}=\mu^u_{\tilde y}. 
\end{equation}
This implies that the induced system on $\mathbb T^d$ is ${\mathrm{Hol}^{cs}}$-invariant.
\end{proof}
\subsection{The $u$-Margulis system as a coresonant state}
In this subsection, we show that the system of Margulis measures constructed in Proposition \ref{prop:umarg} defines a co-resonant state associated to the resonance $\lambda^u(A)$. The proof is the same as in \cite[\S 5.3]{Humu} and we include the proof for completeness.

Let $\varphi \in \Omega_{d_s+d_c}$ supported in a small rectangle $R$. We define a current by
\begin{equation}
\label{eq:current}
\forall q\in R, \quad \mu_q^u(\varphi):=\int_{\mathcal W_R^u(q)}\Big(\int_{\mathcal W^{cs}_R(y)}\varphi\Big)d\mu^u_q(y).
\end{equation}
\begin{lemma}
\label{lemmWF}
The current $\mu^u_q$ is independent of $q\in R$ and thus defines a current on $\mathbb T^d$ which we will denote by $\mu^u$. Moreover, one has
\begin{equation}
\label{eq:WF}
f^*\mu^u=\lambda^u(A) \mu^u, \quad \mathrm{WF}(\mu^u)\subset E_s^*.
\end{equation}
In other words, $\mu^u$ is a co-resonant state associated to the resonance $\lambda^u(A)$.
\end{lemma}
\begin{proof}
Let $p,q\in R$. If $p\in \mathcal W^u_R(q)$, then $\mu^u_p=\mu^u_q$ which implies that $\mu^u_p(\varphi)=\mu^u_q(\varphi)$. If $q\in \mathcal W^{cs}_R(p)$, then using the holonomy invariance \eqref{eq:holhol}, we have
$$ \mu_q^u(\varphi)=\int_{\mathcal W^{u}_R(q)}\Big(\int_{\mathcal W^{cs}_R(y)}\varphi\Big)d\mu^u_q(y)=\int_{\mathcal W_R^u(p)}\Big(\int_{\mathcal W^{cs}_R(z)}\varphi\Big)d\big((\mathrm{Hol}^{cs}_{q,p})_*\mu^u_q\big)(z)=\mu_p^u(\varphi).$$
Using the local product structure, we deduce that $\mu^u_q$ does not depend on the point $q\in R$. Using a partition of unity $(\chi_i)_{i=1}^m$ associated to a finite cover $(R_i)_{i=1}^m$ by rectangles, we can define $\mu^u$ globally on $\mathbb T^d$:
$$\forall \varphi \in \Omega_{d_c+d_s}, \quad \mu^u(\varphi)=\sum_{i=1}^m \mu^u_{p_i}(\chi_i \varphi), \ p_i\in R_i, $$
and the value does not depend on any choice.
 The first equation in \eqref{eq:WF} is a consequence of \eqref{eq:confu}. Indeed, let $\varphi \in \Omega_{d_c+d_s}$ be such that $\varphi$ is supported in $R_i$ and $f_*\varphi$ is supported in $R_j$, then for some $p_i\in R_i$ and $p_j\in R_j$,
 \begin{align*}
 f^*\mu^u(\varphi)&=\mu^u(f_*\varphi)=\int_{\mathcal W_{R_j}^u(p_j)}\Big(\int_{\mathcal W^{cs}_{R_j}(y)}f_*\varphi\Big)d\mu^u_{p_j}(y)=\int_{\mathcal W_{R_j}^u(p_j)}\Big(\int_{f^{-1}(\mathcal W^{cs}_{R_j}(y))}\varphi\Big)d\mu^u_{p_j}(y)
 \\&=\int_{\mathcal W_{R_j}^u(p_j)}\Big(\int_{\mathcal W_{R_j}^{cs}(f^{-1}(y))}\varphi\Big)d\mu^u_{p_j}(y)=\int_{\mathcal W_{R_i}^u(p_i)}\Big(\int_{\mathcal W_{R_i}^{cs}(z)}\varphi\Big)d(f^*\mu^u_{p_j})(y)
 \\&=\lambda^u(A)\int_{\mathcal W_{R_i}^u(p_i)}\Big(\int_{\mathcal W_{R_i}^{cs}(z)}\varphi\Big)d\mu^u_{p_i}(y)=\lambda^u(A)\mu^u(\varphi).
 \end{align*}
 Note as well that $\mu^u$ is not zero since the system of measure has full support by Proposition \ref{prop:umarg}. 
 
 It remains to prove the wavefront set bound. We note that the center-stable manifolds are a priori only $C^1$, hence, we will need to use some smoothness along the stable direction to obtain the desired regularity statement. Consider  a small rectangle $R\subset \mathbb T^d$. Let $\tau^{cs}\in \Lambda^{d_s+d_c}E_{cs}$ be oriented of norm $1$. Define for any Borel set $B\subset R$ the measure
 $$m_R(B):=\int_{\mathcal W^u_R(q)}\mathrm{Vol}_{\mathcal W^{cs}_R(q)}(B\cap  \mathcal W^{cs}_R(y))d\mu^u_q(y).$$
 Then we have for any $\varphi \in \Omega_{d_s+d_c}$ supported in  $R$:
 $$\mu^u(\varphi)=\int_R \langle \varphi, \tau^{cs}\rangle dm_R. $$
Choose a smooth transversal $\Sigma$ to $\mathcal W^s$. We first disintegrate
the Riemannian measure of each center-stable plaque along the strong stable plaques, and
then integrate the corresponding quotient measures against the measure $\mu^u_q$. We obtain a finite measure $\hat{\mu}$ on $\Sigma$ such that for any $\varphi \in \Omega_{d_s+d_c}$ supported in $R$, we have 
 \begin{equation}
 \label{eq:GBGW}
 \mu^u(\varphi)=\int_\Sigma\int_{\mathcal W^s_R(q)}\langle \phi(z),b_q(z)\rangle d\mathrm{Vol}_{\mathcal W^s(q)}(z)d\hat \mu(q),
 \end{equation}
 where $b_q$ is a vector-valued density.  We claim that for any $r\geq 0$, 
 \begin{equation}
 \label{eq:b_q}
 \mathrm{ess \, sup}_{q\in\Sigma}\|b_q\|_{C^r(\mathcal W^s_R(q))}<\infty.
 \end{equation}
 The proof of the claim is an adaptation of \cite[Proposition 1.8]{GBGW}. Indeed, using that $E_u$ dominates $E_{cs}$ as well as the uniform contraction of $E_s$, the backward graph transform implies that $E_{cs}$ (and hence $\tau^{cs}$) is uniformly smooth along $E_s$ in the sense above. This implies the uniform smoothness of $b_q/\|b_q\|$ along $E_s$. The norm of $b_q$ on the other hand is given by the Jacobian of strong stable holonomy between local $d_c$-dimensional transversals inside a center-stable plaque. Let $S(z)$ denote the  $d_c$-plane tangent to the center transversal inside $\Sigma\cap \mathcal W^{cs}$. Denote by $\tilde f$ the lift of $f$ to the Grassmannian, and let 
 $$\mathcal J(x,V):=|\mathrm{det}(df_x|_V)|. $$ Then \cite[Theorem 7.1,
Remark 7.2, and equation (7.3)]{Pe} (see also \cite[Equation (1.12)]{GBGW}) implies that 
 $$\mathrm{Jac}(\mathrm{Hol}^{s})(z)=\prod_{n=0}^{+\infty}\frac{\mathcal J(z,\tilde f^n S(z))}{\mathcal J(\mathrm{Hol}^sz,\tilde f^n S(\mathrm{Hol}^s z))}. $$
 Using the uniform stable contraction and the fact that $E_c$ dominates $E_s$, we deduce that the Jacobian is uniformly smooth along $E_s$, we refer to the proof of \cite[Proposition 1.8]{GBGW} for more details. 
 
 Consider $p\in R$ and  $\xi \notin E_s^*(p)$ and we shall prove that $(p,\xi)$ is not in the wavefront set. Choose a phase function $S$ such that $d_pS=\xi$ and a cutoff $\chi$ near $p$, and compute
 $$\mu^u(\chi e^{iS/h})=\int_\Sigma\int_{\mathcal W^s_R(q)}\langle \chi(z),b_q(z)\rangle e^{iS(z)/h}d\mathrm{Vol}_{\mathcal W^s(q)}(z)d\hat \mu(q). $$
 Since the density is smooth along $E_s$, we can integrate by parts in the stable variable to show that the integrand is $O(h^{\infty})$, that is, for any $N>0$, one can find a constant $C_N>0$ such that it is bounded by $C_Nh^N$, as long as $dS$ does not vanish on $\mathcal W_R^{s}(q)$, which can be ensured near $p$ by the definition of $E_{s}^*$, see \eqref{eq:dual}. This shows that $\xi \notin \mathrm{WF}(\mu^u)$, see for instance \cite[Definition 4.1.2]{Lef}, and thus $ \mathrm{WF}(\mu^u)\subset E_{s}^*$.
\end{proof}

\subsection{Characterizing the $u$-MMEs}
Having constructed a measurable system of $u$-measures, we now characterize $u$-MMEs by their conditionals along the strong-unstable foliation. We recall the following result due to Buzzi, Fisher and Tahzibi \cite{BFT}, see also \cite[Proposition 5.3]{Ta}.
\begin{proposition}
\label{prop:Tah}
Let $f\in \mathrm{Diff}^{\infty}(M)$ be a partially hyperbolic diffeomorphism on a closed manifold $M$ with a measurable Margulis system of measures $(\nu^u_{x})_{x\in M}$ such that 
\begin{itemize}
\item for any $x\in M$, $\nu^u_{x}$ is fully supported in $\mathcal W^u(x)$,
\item there is $D_u>0$ such that for any $x\in M$, 
$ f_*\nu^u_x=D_u^{-1}\nu^u_{f(x)}.$
\end{itemize}
Then for any invariant measure $\nu$, one has
$ h^u_{\nu}(f)\leq \log (D_u),$
with equality if and only if the disintegration of $\nu$ along $\mathcal W^u$ is given by $\nu^u_x$ $ ($up to a constant rescaling$)$, $\nu$-a.e.
\end{proposition}
Together with Propositions \ref{prop1st} and \ref{prop:umarg}, we deduce the following characterization of $u$-MMEs.
\begin{corollary}
\label{corollary}
Let $f\in \mathcal A^\infty_{d_s,d_c,d_u}(\mathbb T^d)$ be as in Proposition \ref{prop:umarg}.  Then
\begin{itemize}
\item we have the equalities
\begin{equation}
\label{eq:allequal}
e^{P(J^c_f)}=e^{h^u_{\mathrm{top}}(f)}=e^{h^u_{\mathrm{top}}(A)}=\lambda^u(A);
\end{equation}
\item as well as $\mu_{J^c_f}\in \mathcal M^u(f)$;
\item and for $\nu\in \mathcal P_f(\mathbb T^d)$, we have $\nu\in \mathcal M^u(f)$ if and only if the conditionals of $\nu$ along $\mathcal W^u$ are given by $(\mu^u_x)_{x\in \mathbb T^d}$ $($up to a constant rescaling$)$, $\nu$-a.e.
\end{itemize}
\end{corollary}
\begin{proof} Note that $f$ is topologically mixing since it is conjugate to its linearization. 
We first prove \eqref{eq:allequal}. From \eqref{eq:ineq}, we have $h^u_{\mathrm{top}}(f)\geq P(J^c_f).$ Next, we use Proposition \ref{prop:umarg}, Proposition \ref{prop:Tah} applied to the system $(\mu^u_x)_{x\in \mathbb T^d}$ (note that the system is holonomy-invariant, hence continuous) and \eqref{eq:varia111} to deduce that $\lambda^u(A)\geq e^{h^u_{\mathrm{top}}(f)}$. Finally, Lemma \ref{lemmWF} implies that $\lambda^u(A)\in \mathrm{Res}_{d_s+d_c}(f)$ and \eqref{eq:equality} implies that $e^{P(J^c_f)}\geq \lambda^u(A)$. In particular, all the inequalities are equalities and we deduce that $e^{P(J^c_f)}=e^{h^u_{\mathrm{top}}(f)}=\lambda^u(A)$. Applying the result to $A$ gives  \eqref{eq:allequal}.

The second point follows from \eqref{eq:allequal} and Lemma \ref{lemm:ineq}. The last point follows from \eqref{eq:allequal} and Proposition  \ref{prop:Tah}.
\end{proof}
\section{Defining the $cs$-system}
\label{sec:cs}
In this section, we construct a $cs$-Margulis system. The $cs$-system could be constructed from the resonant state as in \cite[\S 5]{Humu}. We choose instead to define it more directly using the $u$-MME $\mu_{J^c_f}$ and the $u$-Margulis system $(\mu^u_{x})_{x\in \mathbb T^d}.$ 
\begin{proposition}
\label{prop:cs}
Let $f\in \mathcal A^\infty_{d_s,d_c,d_u}(\mathbb T^d)$ be as in Proposition \ref{prop:umarg}. Then there exists a $cs$-system $(\mu^{cs}_x)_{x\in \mathbb T^d}$ such that:
\begin{itemize}
\item for any $x\in \mathbb T^d$, $\mu^{cs}_x$ has full support in $\mathcal W^{cs}(x)$;
\item one has 
\begin{equation}
\label{eq:confcs}
\forall  x\in \mathbb T^d,\quad f_*\mu^{cs}_x=\lambda^{u}(A)\mu^{cs}_{f(x)};
\end{equation}
\item the system is $\mathrm{Hol}^{u}$-invariant. For any rectangle $R\subset \mathbb T^d$ and any $p,q\in R$, we have $(\mathrm{Hol}^{u}_{p,q})_*\mu^{cs}_p=\mu^{cs}_q.$
\end{itemize}
The $u$-system from Proposition \ref{prop:umarg} and the $cs$-system $(\mu^{cs}_x)_{x\in \mathbb T^d}$ satisfy \eqref{eq:localproduct}.
\end{proposition}
\begin{proof}
The idea is to construct $(\mu^{cs}_x)_{x\in \mathbb T^d}$ as the $cs$-conditionals of the $u$-MME $\mu_{J^c_f}$. More precisely, we lift everything to the universal cover and fix a base point $\tilde x_0\in \mathbb R^d$. For $Z\subset \widetilde{\mathcal W^u_f}(\tilde x_0)$ a non-empty bounded open set, we define
\begin{equation}
\label{eq:mucsZ}
\forall \tilde x\in \mathbb R^d, \ \forall Y\subset \widetilde{\mathcal W^{cs}_f}(\tilde x),\quad  \mu_{Z,\tilde x}^{cs}(Y):=\frac{\tilde \mu_{J^c_f}([Z,Y]_{cs,u})}{\mu^u_{\tilde x_0}(Z)},
\end{equation}
where $Y$ is any bounded Borel subset of $\widetilde{\mathcal W^{cs}_f}(\tilde x)$,   $\tilde \mu_{J^c_f}$ denotes the lift to $\mathbb R^d$ of $\mu_{J^c_f}$ and 
$$[Z,Y]_{cs,u}:=\{[z,y]_{cs,u}\mid y\in Y, \ z\in Z\},$$
denotes the Bowen bracket of $Y$ and $Z$. Note that \eqref{eq:mucsZ} is well defined since $\mu^u_{\tilde x_0}$ is fully supported. 
\\
\textbf{The system $(\mu_{Z,\tilde x}^{cs})_{\tilde x\in \mathbb R^d}$ is fully supported.}
Since $\mu_{J^c_f}$ is a probability measure, its lift $\tilde \mu_{J^c_f}$ is finite on bounded subsets of $\mathbb R^d$ and this implies that $\mu_{Z,\tilde x}^{cs}$ is finite on bounded subsets of $\widetilde{\mathcal W^{cs}_f}(\tilde x)$. Finally, since $\mu_{J^c_f}$ is an equilibrium state, it has full support and we hence deduce from \eqref{eq:mucsZ} that for any non-empty bounded open subset $Y\subset\widetilde{\mathcal W^{cs}_f}(\tilde x)$, we have $\mu_{Z,\tilde x}^{cs}(Y)>0$.

%From the definition of $\tilde \mu^u_{J^c_f}$, \eqref{eq:mucsZ} and the first step of the proof of Proposition \ref{prop:umarg}, the system $(\mu_{Z,\tilde x}^{cs})_{\tilde x\in \mathbb R^d}$ descends to a system $(\mu_{Z, x}^{cs})_{ x\in \mathbb T^d}$ on $\mathbb T^d$. As in the proof of Proposition \ref{prop:umarg}, it suffices to prove all the claimed properties of the $cs$-system on the universal cover, which we now do.
\textbf{The system $(\mu_{Z,\tilde x}^{cs})_{\tilde x\in \mathbb R^d}$ does not depend on $Z$.} 
Consider a non-empty bounded open set $Z'\subset \widetilde{\mathcal W^u_f}(\tilde x_0)$. First assume that $Z'\subset Z$. Consider $\tilde x\in \mathbb R^d$ and let $Y\subset \widetilde{\mathcal W^{cs}_f}(\tilde x)$ denote a bounded Borel set. Consider $R\subset \mathbb R^d$ a rectangle containing $[Z,Y]_{cs,u}$, then by Corollary \ref{corollary} (using that rectangles in $\mathbb R^d$ which are small enough project to rectangles on $\mathbb T^d$ as well as the compatibility of conditionals between different rectangles), we see that there exists a $\hat \mu_{R,\tilde x}$ a measure on $\widetilde{\mathcal W^{cs}_{f,R}}(\tilde x)$ such that 
\begin{equation}
\label{eq:CPZ2}
\forall \phi\in L^1(R,\tilde\mu_{J^c_f}),\quad \tilde\mu_{J^c_f}(\phi)=\int_{\widetilde{\mathcal W^{cs}_{f,R}}(\tilde x)}\int_{\widetilde{\mathcal W^{u}_{f,R}}(\tilde y)}\phi(\tilde z)d\mu^{u}_y(\tilde z)d \hat \mu_{R,\tilde x}(\tilde y).
\end{equation}
Choose $\phi=\mathbf{1}_{[Z',Y]_{cs,u}}$. Then using the holonomy invariance of $(\mu^u_{\tilde x})_{x\in \mathbb R^d}$ obtained in the last step of the proof of Proposition \ref{prop:umarg}, we obtain
$$ \tilde \mu_{J^c_f}(\mathbf{1}_{[Z',Y]_{cs,u}})=\tilde \mu_{J^c_f}({[Z',Y]_{cs,u}})=\mu^u_{\tilde x_0}(Z')\hat \mu_{R,\tilde x}(Y).$$
This implies that 
\begin{equation}
\label{eq:RRR}\mu_{Z',\tilde x}^{cs}(Y)=\hat \mu_{R,\tilde x}(Y)=\mu_{Z,\tilde x}^{cs}(Y).
\end{equation}
For a general $Z'$, we observe that $Z'':=Z'\cup Z$ is a non-empty open subset of $\widetilde{\mathcal W^u_f}(\tilde x_0)$ containing both $Z$ and $Z'$ and applying the result above gives
$$\mu_{Z',\tilde x}^{cs}(Y)=\mu_{Z'',\tilde x}^{cs}(Y)=\mu_{Z,\tilde x}^{cs}(Y).$$
This shows that for a fixed base point $\tilde x_0$, the system of measure does not depend on the choice of open set $Z\subset \widetilde{\mathcal W^u_f}(\tilde x_0)$. Consider now another base point $\tilde x_1$. Let $Z':=\mathrm{Hol}^{cs}_{\tilde x_0,\tilde x_1}(Z)\subset \widetilde{\mathcal W^u_f}(\tilde x_1)$. Then $Z'$ is a non-empty open set and by \eqref{eq:Holcsuniv}, one has $\mu^u_{\tilde x_1}(Z')=\mu^u_{\tilde x_0}(Z)$. For any $\tilde x\in \mathbb R^d$ and $Y\subset \widetilde{\mathcal W^{cs}_f}(\tilde x)$ Borel subset, one has $[Z,Y]_{cs,u}=[Z',Y]_{cs,u}$ which implies that $(\mu_{Z,\tilde x}^{cs})_{\tilde x\in \mathbb R^d}=(\mu_{Z',\tilde x}^{cs})_{\tilde x\in \mathbb R^d}$. This shows that the system does not depend on the base point. We will now write $(\mu_{\tilde x}^{cs})_{\tilde x\in \mathbb R^d}$ for this system of measures. 
\\ \textbf{The system $(\mu^{cs}_{\tilde x})_{x\in \mathbb R^d}$ descends to $\mathbb T^d$.} 
Let $\gamma \in \mathbb Z^d$ and let $\tilde x+\gamma$ be another lift of some $x\in \mathbb T^d$. Recall from \eqref{eq:Tgamma} that one has $(T_\gamma)_*\mu^u_{\tilde x}=\mu^u_{\tilde x+\gamma}$. By definition, one has $(T_\gamma)_*\tilde{\mu}_{J^c_f}=\tilde{\mu}_{J^c_f}$ and we hence deduce that $(T_\gamma)_*\mu^{cs}_{\tilde x}=\mu^{cs}_{\tilde x+\gamma}$ from \eqref{eq:mucsZ} and the fact that the system does not depend on the choice of $Z$.
As in the proof of Proposition \ref{prop:umarg}, it suffices to prove all the claimed properties of the $cs$-system on the universal cover, which we now do.
\\ \textbf{The system $(\mu_{\tilde x}^{cs})_{\tilde x\in \mathbb R^d}$ satisfies \eqref{eq:confcs}.} Let $\tilde x\in \mathbb R^d$ and $Y\subset \widetilde{\mathcal W^{cs}_f}(F\tilde x)$ denote a bounded Borel set. Then using \eqref{eq:CONF} and the $F$-invariance of $\tilde \mu_{J^c_f}$, we compute
\begin{align*}
F_*\mu_{\tilde x}^{cs}(Y)&=\mu_{\tilde x}^{cs}(F^{-1}Y)=\frac{\tilde \mu_{J^c_f}([Z,F^{-1}Y]_{cs,u})}{\mu^u_{\tilde x_0}(Z)}=\frac{\tilde \mu_{J^c_f}(F^{-1}[FZ,Y]_{cs,u})}{\mu^u_{\tilde x_0}(Z)}
\\&=\frac{\tilde \mu_{J^c_f}([FZ,Y]_{cs,u})}{\mu^u_{\tilde x_0}(Z)}=\lambda^u(A)\frac{\tilde \mu_{J^c_f}([FZ,Y]_{cs,u})}{\mu^u_{F(\tilde x_0)}(FZ)}=\lambda^u(A)\mu_{F(\tilde x)}^{cs}(Y).
\end{align*}
\textbf{The system $(\mu_{\tilde x}^{cs})_{\tilde x\in \mathbb R^d}$ is $\mathrm{Hol}^u$-invariant.}
Let $\tilde x,\tilde y\in \mathbb R^d$ and $Y\subset \widetilde{\mathcal W^{cs}_f}(\tilde x)$ denote a bounded Borel set. Then using \eqref{eq:mucsZ} and $[Z,Y]_{cs,u}=[Z,\mathrm{Hol}^{u}_{\tilde x,\tilde y}Y]_{cs,u}$, we have
$$\mu_{\tilde x}^{cs}(Y)=\frac{\tilde \mu_{J^c_f}([Z,Y]_{cs,u})}{\mu^u_{\tilde x_0}(Z)}=\frac{\tilde \mu_{J^c_f}([Z,\mathrm{Hol}^{u}_{\tilde x,\tilde y}Y]_{cs,u})}{\mu^u_{\tilde x_0}(Z)} =\mu_{\tilde y}^{cs}(\mathrm{Hol}^{u}_{\tilde x,\tilde y}Y).$$
\textbf{The systems satisfy \eqref{eq:localproduct}.}
This follows directly from \eqref{eq:CPZ2} and \eqref{eq:RRR}.
\end{proof}
\section{Proofs of Corollaries \ref{cor1}, \ref{cor2} and \ref{cor3}.}
\label{sec:last}
In this last section, we prove several corollaries of our main results. We consider $f\in \mathcal A^{\infty}_{d_s,d_c,d_u}(\mathbb T^d)$ satisfying the hypothesis of Theorem \ref{main theo}. By \cite[Theorem 20.3.9]{HK}, $\mu_{J^c_f}$ is a measure of maximal entropy if and only if $J^c_f $ is cohomologous to a constant $\lambda_c$. By the Livsic theorem \cite[Theorem 19.2.1]{HK}, this  holds if and only if for any periodic point $p$ of period $n(p)$, one has
$$\frac1{n(p)}\sum_{k=0}^{n(p)-1} J^c_f(f^kp)=\lambda_c.$$
We now show that in this case, the constant $\lambda_c$ has to be equal to $-\log \lambda^c(A)$ where $\lambda^c(A):=|\mathrm{det}(A|_{E^A_c})|$. Indeed, by \eqref{eq:equality}, one has $\log\lambda^u(A)=P(J^c_f)$ and $h_{\mathrm{top}}(f)=h_{\mathrm{top}}(A)=\log \lambda^u(A)+\log \lambda^c(A)$. By elementary properties of the pressure functional, we deduce
$$\log\lambda^u(A)=P(J^c_f)=P(\lambda_c)=h_{\mathrm{top}}(f)+\lambda_c= \log \lambda^u(A)+\log \lambda^c(A)+\lambda_c.$$
This implies that $\lambda_c=-\log \lambda^c(A)$. We start with the proof of Corollary \ref{cor1}.
\begin{proof}[Proof of Corollary \ref{cor1}]
Let $A\in \mathrm{SL}(d,\mathbb Z)\cap \mathcal A^\infty_{d_s,d_c,d_u}(\mathbb T^d)$ satisfying Assumption \ref{assump}. We note that any $f\in \mathcal A^\infty_A(\mathbb T^d)$ also satisfies Assumption \ref{assump} and the previous discussion applies. We define 
$$X:=\{f\in \mathcal A^\infty_A(\mathbb T^d)\mid \mu_{J^c_f} \text{ is not a MME for } f\}\subset \mathcal A^\infty_A(\mathbb T^d).$$
By the previous discussion, $f\in X$ if and only if there is a periodic point $p$ of period $n(p)$ such that:
$$\frac1{n(p)}\sum_{k=0}^{n(p)-1} J^c_f(f^kp)\neq -\log \lambda^c(A).$$
This condition is clearly $C^1$-open in $f$ which shows that $X$ is $C^1$-open. 

We now show that it is $C^\infty$-dense. Let $f\in \mathcal A^\infty_{d_s,d_c,d_u}(\mathbb T^d)$, since $X$ is $C^1$-open, we can suppose without loss of generality that $f\notin X$. We will now construct $C^\infty$ perturbations of $f$ which lie in $X$. Choose a periodic point $p$ of $f$ of period $n$. Let $U\subset \mathbb T^d$ be an open neighborhood of $p$ which does not contain any $f^ip$ for $i=1,2,\ldots, n-1$. For $\eps>0$, we let $A_\eps:T_p\mathbb T^d\to T_p\mathbb T^d$ be a linear map close to the identity preserving the splitting $ T_p\mathbb T^d=E_s^f(p)\oplus E_c^f(p)\oplus E_u^f(p)$ and such that $\mathrm{det}(A_\eps|_{E_c^f(p)})=e^\eps$. We let $\varphi_\eps: \mathbb T^d\to \mathbb T^d$ be a diffeomophism supported in $U$ such that 
$$\varphi_\eps(p)=p,\ d_p\varphi_\eps=A_\eps, \ \varphi_\eps\to \mathrm{Id} \text{ in } C^\infty. $$
We let $f_\eps:=f\circ \varphi_\eps$. Then $f_\eps \to f$ in $C^\infty$ and we now show that if $\eps>0$ is small enough, we have $f_\eps \in X$. Since $f\notin X$, the discussion at the beginning of the section implies that 
$$\frac1{n}\sum_{k=0}^{n-1} J^c_f(f^kp)=-\frac 1n \log |\mathrm{det}(df^n|_{E_c^f(p)})|=-\log \lambda^c(A), $$
where we used the chain rule for the first equality. Since $\varphi_\eps(p)=p$ and $\varphi_\eps$ is supported in $U$, the orbit of $p$ by $f_\eps$ is the same as the orbit of $p$ for $f$. Moreover, we have
$$\frac1{n}\sum_{k=0}^{n-1} J^c_{f_\eps}(f^k_\eps p)=-\frac 1n \log |\mathrm{det}(df^n_\eps|_{E_c^f(p)})|=-\log \lambda^c(A)-\frac \eps n,$$
where we have used that $|\mathrm{det}(df^n_\eps|_{E_c^f(p)})|=e^\eps |\mathrm{det}(df^n|_{E_c^f(p)})|$ since the diffeomorphism $\varphi_\eps$ is supported in $U$.
For $\eps>0$ small enough, this means that $f_\eps$ does not respect the periodic orbit condition and this shows that $f_\eps\in X$. This concludes the proof of the corollary.
\end{proof}
 We now show Corollary \ref{cor2}.
 \begin{proof}[Proof of Corollary \ref{cor2}]
Let $A\in \mathrm{SL}(d,\mathbb Z)\cap \mathcal A^\infty_{d_s,1,d_u}(\mathbb T^d)$ with a $1$-dimensional center bundle. Suppose furthermore that $A$ is generic and satisfies Assumption \ref{assump} . This allows us to apply \cite[Theorem 1.3]{GS23} to $f^{-1}$ and this shows that there is a $C^1$-neighborhood $\mathcal U$ of $A$ such that for any $f\in \mathcal U$, the fact that the center Lyapunov exponent of $f$ is equal to that of its linearization $A$ is equivalent to $E_s$ and $E_u$ being jointly integrable. This concludes the proof of the corollary by the discussion at the beginning of the section and the fact that any nearby $f$ satisfies Assumption \ref{assump}. 
 \end{proof}
We now show Corollary \ref{cor3}.
 \begin{proof}[Proof of Corollary \ref{cor3}]
 Under the assumptions of the corollary and from Theorem \ref{main theo}, we see that we can apply \cite[Theorem A]{GS25} which gives that $\mu_{J^c_f}$ is a $u$-Gibbs measure if and only if $f$ has constant unstable periodic data. This last statement is equivalent to $J^u_f$ being cohomologous to a constant.
 \end{proof}
 We finish this section by giving examples of $f \in A^\infty_A(\mathbb T^d)$ for some $A\in \mathrm{SL}(d,\mathbb Z)\cap \mathcal A^\infty_{d_s,d_c,d_u}(\mathbb T^d)$ which have minimal strong-unstable foliation for which Corollary \ref{cor3} applies. We start from a $A\in \mathrm{SL}(d,\mathbb Z)\cap \mathcal A^\infty_{d_s,d_c,1}(\mathbb T^d)$ which is irreducible and has a simple leading eigenvalue $\lambda_u$ whose eigenspace is spanned by $v_u$. Let $\psi:\mathbb T^d\to \mathbb R$ be a smooth function which is non-constant. We let 
 $$ \tau_\psi:\mathbb T^d\to \mathbb T^d, \ x\mapsto x+\psi(x)v_u.$$
We then define $f_\psi:=\tau_\psi\circ A$. If $\psi$ is  $C^1$-small, then $f_\psi\in \mathcal A^\infty_{d_s,d_c,1}(\mathbb T^d)$. Moreover, we have 
$$\forall x\in \mathbb T^d, \quad d_xf_\psi(v_u)=\lambda_u(1+d_{Ax}\psi(v_u))v_u. $$
This shows that $\mathcal W^u_{f_\psi}=E_u^A$ and is therefore minimal by \cite[Proposition 6]{Gog08}. Since $\psi$ is not constant, the diffeomorphism is not linear.
\bibliography{biblio}
\bibliographystyle{alpha}

\end{document}